\documentclass{amsart}
\usepackage{amsmath,amssymb}
\usepackage[left=1.5in,right=1.5in,bottom=1in,top=1in]{geometry}
\usepackage[all]{xy}
\usepackage[color=cyan]{todonotes}
\usepackage{mathpazo,euler,eucal}
\usepackage{multicol}

\makeatletter
\providecommand\@dotsep{5}
\renewcommand{\listoftodos}[1][\@todonotes@todolistname]{%
  \@starttoc{tdo}{#1}}
\makeatother

\newcommand{\ord}{\mathrm{ord}}
\newcommand{\A}{\mathbb{A}}
\newcommand{\F}{\mathbb{F}}
\newcommand{\N}{\mathbb{N}}
\newcommand{\M}{\mathbb{M}}
\newcommand{\Q}{\mathbb{Q}}
\newcommand{\Z}{\mathbb{Z}}
\newcommand{\C}{\mathbb{C}}

\newcommand{\Aut}{\mathrm{Aut}}
\newcommand{\Hom}{\mathrm{Hom}}
\newcommand{\End}{\mathrm{End}}

\newcommand{\SL}{\mathrm{SL}}
\newcommand{\SO}{\mathrm{SO}}
\newcommand{\GL}{\mathrm{GL}}

\newcommand{\GSp}{\mathrm{GSp}}
\newcommand{\GU}{\mathrm{GU}}

\newcommand{\Spec}{\mathrm{Spec}}

\newcommand{\G}{\mathbb{G}}

\newcommand{\Nilp}{\mathrm{Nilp}}

\begin{document}

\newtheorem*{theoremA*}{Theorem A}
\newtheorem*{theoremB*}{Theorem B}
\newtheorem*{theoremC*}{Theorem C}
\newtheorem{definition}{Definition}[section]
\newtheorem{theorem}[definition]{Theorem}
\newtheorem{claim}[definition]{Claim}
\newtheorem{example}[definition]{Example}
\newtheorem{lemma}[definition]{Lemma}
\newtheorem{corollary}[definition]{Corollary}
\newtheorem{proposition}[definition]{Proposition}
\newtheorem{notation}[definition]{Notation}

\setcounter{tocdepth}{1}

\vspace{4cm}

\begin{center}
{\huge\textbf{The $\GL_4$ Rapoport-Zink Space} }\\
\end{center}

\vspace{1cm}

\begin{center}
{\Large{Maria Fox} } \\
\end{center}

\vspace{1cm}

\begin{minipage}[c]{5in}
ABSTRACT: We give a description of the $\GL_4$ Rapoport-Zink space, including the connected components, irreducible components, intersection behavior of the irreducible components, and Ekedahl-Oort stratification.  As an application of this, we also give a description of the supersingular locus of the Shimura variety for the group $\GU(2,2)$ over a prime split in the relevant imaginary quadratic field.
\end{minipage}

\tableofcontents

\section{Introduction}

This paper contributes to the theory of Rapoport-Zink formal schemes by giving an explicit description of the  $\GL_4$ Rapoport-Zink space. As an application of the main result, this paper also provides a description of the supersingular locus of the $\GU(2,2)$ Shimura variety at a prime split in the relevant quadratic imaginary field.

The $\GL_4$ Rapoport-Zink space is a moduli space of pairs $(G, \rho)$, where $G$ is a supersingular $p$-divisible group of height 4 and dimension 2, and $\rho$ is a quasi-isogeny from $G$ to a fixed basepoint $p$-divisible group $\mathbb{G}$ (see Section ~\ref{defRZ} for a precise definition). By Rapoport and Zink \cite{RZ}, this moduli problem is representable by a formal scheme. By a result of Viehmann ~\cite{Viehmann}, the reduced $\overline{\F}_p$-scheme underlying the $\GL_4$ Rapoport-Zink space is 1-dimensional,  with connected components given by the height of the quasi-isogeny. Further, a reader familiar with the literature might expect (based on work of Katsura-Oort ~\cite{KatO} and Kudla-Rapoport ~\cite{KR} on the  $\GSp_4$ Rapoport-Zink space, for example) some of the irreducible components of the $\GL_4$ Rapoport-Zink space to be projective lines. However, this is far from a complete description: there are similarly defined Rapoport-Zink spaces for other groups, and a typical analysis  will at minimum describe the irreducible components, the intersection behavior of the irreducible components, and the connected components (see  ~\cite{VW}, ~\cite{GU22}, and \cite{RTW} for some examples, among many others). The main result of this paper is the following description of the $\GL_4$ Rapoport-Zink space:

\begin{theoremA*}
Let $p$ be an odd prime. The $\GL_4$ Rapoport-Zink space $\mathcal{N}$ decomposes as a disjoint union:
$$\mathcal{N} = \bigsqcup_{i \in \Z} \mathcal{N}_i$$
where $\mathcal{N}_i$ is the locus of points $(G, \rho)$ where the quasi-isogeny $\rho$ is of height $i$. The components $\mathcal{N}_i$ are all isomorphic, as formal schemes over $\mathrm{Spf}(W(\overline{\F}_p) )$.

Let $\mathcal{N}_{0,red}$ be the reduced $\overline{\F}_p$-scheme underlying $\mathcal{N}_0$. Then $\mathcal{N}_{0,red}$ is connected, and decomposes as:
$$\mathcal{N}_{0,red} = \bigcup_{\Lambda} \mathcal{N}_{\Lambda}$$
where the index set is the collection of all vertex lattices $\Lambda$ in $V^{\Phi}$ of type 4. These $\mathcal{N}_{\Lambda}$ are precisely the irreducible components of $\mathcal{N}_{0,red}$, and each is isomorphic, over $\overline{\F}_p$, to $\mathbb{P}^1$.

Two irreducible components are either disjoint or intersect in a single point.  Each  irreducible component contains $p^2 + 1$ intersection points, and each intersection point is the intersection of $p^2 + 1$ irreducible components. The intersection points are exactly the superspecial points of $\mathcal{N}_{0,red}$, and they are parametrized by the vertex lattices in $V^{\Phi}$ of type 2.

Further, $\mathcal{N}_{red}$ (the reduced $\overline{\F}_p$-scheme underlying $\mathcal{N}$) has two Ekedahl-Oort strata: one consisting of the superspecial points, and the other the complement of the superspecial points.
 \end{theoremA*}
 
  (In the description above, $V^{\Phi}$ is a certain quadratic space of dimension 6 over $\Q_p$. See Section ~\ref{Vertexsection} for the definitions of $V^{\Phi}$ and of a vertex lattice.) It is worth mentioning that, in terms of Rapoport-Zink spaces with supersingular basepoints, after the $\GL_2$ Rapoport-Zink space, which is zero-dimensional, and the $\GL_3$ Rapoport-Zink space, which is empty, the $\GL_4$ Rapoport-Zink space has the most fundamental moduli description, yet relatively little of its specific geometry had been explored so far. Further, the parametrization given in this paper of the $\GL_4$ Rapoport-Zink space by a union of projective lines is quite explicit.

The methods of this paper were inspired by the analysis by Howard-Pappas in ~\cite{GU22} of the  $\GU(2,2)$ Rapoport-Zink space (a moduli space of supersingular $p$-divisible groups of height 8 and dimension 4, with principal polarization and an action of an imaginary quadratic field, subject to some requirements) when the prime $p$ is inert in the relevant imaginary quadratic field. Both the results of Howard-Pappas and the results in this paper rely upon a form of the exceptional isomorphism $\mathrm{SU}(2,2) \cong \mathrm{Spin}(4,2)$, which is why a quadratic space appears in Theorem A above.  The work of Vollaard ~\cite{V} and Vollaard-Wedhorn ~\cite{VW} on  the $\GU(n,1)$ Rapoport-Zink spaces when the prime $p$ is inert was also very influential, as was the description given by Rapoport-Terstiege-Wilson \cite{RTW} of the $\GU(n,1)$ Rapoport-Zink spaces in the case that the prime $p$ is ramified, and the description by Howard-Pappas ~\cite{GSpin} of Rapoport-Zink spaces associated to spinor groups. The motivation for the approach taken in this paper is that, if the prime $p$ is split in the imaginary quadratic field, the groups $\GU(2,2)$ and $\GL_4$ over $\Q_p$ are closely related (see Section ~\ref{excepsect} for the precise statement). And the Rapoport-Zink space for $\GU(2,2)$ when the prime $p$ is split will clearly be related to the  Rapoport-Zink space for $\GU(2,2)$ when the prime $p$ is inert. Combining these, it is reasonable to attempt to apply the techniques of Howard-Pappas to study the $\GL_4$ Rapoport-Zink space. However, the relatively small differences in these groups can cause large differences in the geometry of the Rapoport-Zink spaces, and the methods of ~\cite{GU22} required considerable adaptation to compensate for this. 

In this paper we observe that when the prime $p$ is split in the relevant imaginary quadratic field, the $\GU(2,2)$ Rapoport-Zink space decomposes into infinitely many copies of the $\GL_4$ Rapoport-Zink space. So, as a corollary to our main result, we also produce a description of the $\GU(2,2)$ Rapoport-Zink space over a split prime:

\begin{theoremB*}
Let $\mathcal{N}_{GU(2,2)}$ be the $\GU(2,2)$ Rapoport Zink space, and assume that the odd prime $p$ is split in the imaginary quadratic field $E$. The underlying reduced $\overline{\F}_p$-scheme $\mathcal{N}_{GU(2,2), red}$ has connected components naturally indexed by $\Z \times \Z$, and the connected components are all isomorphic. 

Every connected component of $\mathcal{N}_{GU(2,2), red}$  is isomorphic, over $\overline{\F}_p$, to a union of projective lines. Every pair of projective lines is either disjoint or intersects in a single point. Each projective line contains $p^2 + 1$ intersection points, and each intersection point is the intersection of $p^2 + 1$ projective lines. These intersection points are precisely the superspecial points.

Further, $\mathcal{N}_{GU(2,2),red}$ has two Ekedahl-Oort strata: one consisting of the superspecial points, and the other the complement of the superspecial points.

\end{theoremB*}

Finally, as an application of our main result, we make the following contribution to the theory of supersinguar loci of the reduction modulo p of canonical integral models of Shimura varieties: let $E$ continue to be an imaginary quadratic field, let $p$ continue to be an odd prime split in $E$, and let $\mathcal{O}$ be the integral closure of $\mathbb{Z}_{(p)}$ in $E$. Fix a free $\mathcal{O}$-module $V$ of rank 4 with a perfect $\mathcal{O}$-valued Hermitian form of signature $(2,2)$. Let $G = \GU(V)$, fix a compact open subgroup $K^p$ of $\G(\A_f^p)$, and let $K_p = G(\Z_p)$ and $K = K_pK^p$. For sufficiently small $K^p$, there is a scheme $\mathcal{M}_{K}$ over $\Z_{(p)}$ which is a moduli space of four-dimensional abelian varieties with a principal polarization, an action of $\mathcal{O}$, and with level structure of a type determined by $K$,  subject to some constraints (see Section ~\ref{SV} for details). This is the $\GU(2,2)$ Shimura variety.

 The reduced locus of the geometric special fiber of  $\mathcal{M}_{K}$ where the corresponding abelian varieties are supersingular is called the supersingular locus. We use the Rapoport-Zink uniformization theorem to give a description of the supersingular locus of this Shimura variety when the prime $p$ is split in the imaginary quadratic field $E$:

 \begin{theoremC*}
 Let $\mathcal{M}_K^{ss}$ be the supersingular locus of the $\GU(2,2)$ Shimura variety at an odd prime $p$ split in the relevant imaginary quadratic field, with level structure given by $K = K_pK^p$. The $\overline{\F}_p$-scheme $\mathcal{M}_K^{ss}$ is of pure dimension 1. 
 
 For $K^p$ sufficiently small, all irreducible components of $\mathcal{M}_K^{ss}$ are isomorphic, over $\overline{\F}_p$, to $\mathbb{P}^1$. Any two irreducible components either intersect trivially or intersect in a single point.
  
Each irreducible component contains $p^2 + 1$ intersection points, and each intersection point is the intersection of $p^2 + 1$ irreducible components. These intersection points are precisely the superspecial points.

Further, $\mathcal{M}_K^{ss}$ has two Ekedahl-Oort strata: one consisting of the superspecial points, and the other the complement of the superspecial points.
 \end{theoremC*}

\subsection*{Acknowledgements} I would like to thank  Ben Howard for suggesting this problem, for answering numerous questions, and for his support. I would also like to thank  Brian Lehmann, Keerthi Madapusi Pera, Mark Reeder, Ari Shnidman, and Cihan Soylu for very helpful conversations.

\subsection{Notation}

Throughout this paper, $p$ is an odd prime number, $k$ is an algebraically closed field of characteristic $p$, and $k'$ is an arbitrary field extension of $k$. The absolute Frobenius on $k$ will be denoted by $\sigma$.

\section{From the  Rapoport-Zink Space to Dieudonn\'e Lattices}

Over an algebraically closed field, one can study $p$-divisible groups using the theory of Dieudonn\'e modules.  Unforuntately, to understand the structure of the $\GL_4$ Rapoport-Zink space we will need to study $p$-divisible groups over any field extension $k'$ over $k$. Similar to ~\cite{GU22} and ~\cite{GSpin}, this may be accomplished by using Zink's theory of windows. The objective of this section is to introduce the $\GL_4$ Rapoport-Zink space and to use Zink's theory of windows to convert the $k'$-points of our Rapoport-Zink space to a collection of lattices in a fixed vector space.

\subsection{Rapoport-Zink Spaces}\label{defRZ}

Recall that $k$ is an algebraically closed field of characteristic $p$. Let $W = W(k)$ be the ring of Witt vectors of $k$, which is a complete discrete valuation ring with maximal ideal generated by $p$ and residue field $k$. There is a unique lift of the Frobenius to a ring automorphism of $W$, which will also be denoted $\sigma$.  Let $W_{\Q}$ be the field of fractions of $W$, and let $\Nilp_W$ be the category of $W$-schemes on which $p$ is locally nilpotent. 

Fix a basepoint supersingular $p$-divisible group $\mathbb{G}$ over $k$ of height 4 and dimension 2. Our objective is to study the functor:
$$\mathcal{N}: \Nilp_W \rightarrow \mathrm{Sets}$$
that sends  a $W$-scheme $S$ on which $p$ is locally nilpotent to the set $\mathcal{N}(S)$ of pairs $(G, \rho)$, where $G$ is a supersingular $p$-divisible group over $S$ of height 4 and dimension 2, and 
$$\rho: G_{\mathcal{S}_0} \rightarrow \mathbb{G}_{\mathcal{S}_0}$$
is a quasi-isogeny. Here, $\mathcal{S}_0 = S \times_{\Spec(W) } \Spec(k)$.

Two pairs $(G_1, \rho_1)$ and $(G_2, \rho_2)$ over $S$ are isomorphic if there is an isomorphism $\varphi$ from $G_1$ to $G_2$ carrying $\rho_2$ to $\rho_1$: $\rho_1 = \rho_2 \circ \varphi_{S_0}$.  By ~\cite{RZ} Theorem 2.16, $\mathcal{N}$ is represented by a formal scheme over $W$ which is locally formally of finite type over $W$.

\subsection{Dieudonn\'e Lattices}

We will now convert the field-valued points of the $\GL_4$ Rapoport-Zink space to a collection of lattices in a certain vector space. Over $k$, this vector space will be the isocrystal attached to the basepoint $p$-divisible group:

Let $\mathbb{M}$ be the Dieudonn\'e module associated to $\G$, so $\M$ is a free $W$-module of rank 4, and let $\N = \M \otimes_{W} W_{\Q}$ be the corresponding isocrystal. Because $\G$ is supersingular, by the structure theorems for isocrystals (see for example ~\cite{Demazure}), there is a basis $\{e_i\}_1^4$ of $\N$ such that:
  $$Fe_1 = e_2 \quad Fe_2 = pe_1 \quad F e_3 = e_4 \quad F e_4 = p e_3 .$$

For a field $k'$ over $k$, let $W'$ be the Cohen ring of $k'$ (in particular, if $k'=k$, $W' = W$, the Witt vectors of $k$.) The injection $k \rightarrow k'$ gives an injection $W \rightarrow W'$. Note that $W'$ is a complete discrete valuation ring with uniformizer $p$, and $W'/pW' \cong  k'$. Let $W'_{\Q}$ be the fraction field of $W'$, and let $\mathbb{M}' = \mathbb{M} \otimes_{W} W'$ and $\mathbb{N}' = \mathbb{N} \otimes_{W_{\Q} } W'_{\Q}$. 

There is a unique continuous ring homomorphism $W' \rightarrow W'$ reducing to the Frobenius on $k'$, which will also be denoted $\sigma$, and the operator $F$ on $\mathbb{N}$ has a unique $\sigma$-semilinear extension to $\mathbb{N}'$. Note that $F$ is not necessarily surjective on $\mathbb{N}'$, and if $M \subset \mathbb{N}'$ is a $W'$-module, $F^{-1}(M)$ has the structure of a  $W'$-module, but $F(M)$ will not necessarily have the structure of a $W'$-module. So, given a $W'$-module $M \subset \mathbb{N}'$, we will use the notation $\overline{F}(M)$ to denote the $W'$-module generated by $F(M)$.

\begin{definition}\label{DDLattice}
A \emph{Dieudonn\'e lattice} in $\mathbb{N}'$ is a $W'$-lattice $A \subset \mathbb{N}'$ such that:
\begin{enumerate}
\item{ $pA \subset F^{-1}(pA) \subset A$}
\item{ $A = \overline{F}(F^{-1}(A) )$}
\end{enumerate}
If $A \subset \N'$ is a $W'$-lattice, we will use the notation:
$$A_1 = F^{-1}(pA).$$
\end{definition}

Note that, as in ~\cite{GU22}, we may replace condition (2) above with the condition:
$$\mathrm{dim}_{k'}(A_1/pA) = 2.$$

\begin{definition}
For any field $k'$ over $k$, let $\mathcal{M}(k')$ denote the collection of all Dieudonn\'e lattices in~$\mathbb{N}'$.
\end{definition}

\begin{proposition}\label{dd1}
There is a bijection:
$$\mathcal{N}(k') \longleftrightarrow  \mathcal{M}(k').$$
\end{proposition}

The proof of this will require Zink's theory of windows. First, note that $W'$, together with $\sigma$ and the reduction map $W' \mapsto W'/pW' = k'$ is a \emph{frame} for $k'$, in the terminology ~\cite{Windows}. Following ~\cite{GSpin}, the definition given of a Dieudonn\'e window in ~\cite{Windows} may be simplified to:

\begin{definition}
A \emph{Dieudonn\'e $W'$-window over $k'$} consists of the following data:
\begin{enumerate}
\item{A finitely generated, free $W'$-module $M$}
\item{A submodule $M_1 \subset M$ such that $pM \subset M_1 \subset M$}
\item{A $\sigma$-semi-linear map $\Psi: M \rightarrow M$ such that $\Psi(M_1) \subset pM$ and $\overline{\Psi}(p^{-1}M_1) = M$ (here  $\overline{\Psi}(p^{-1}M_1)$ is the $W'$ module generated by  $\Psi(p^{-1}M_1) )$ . }

A \emph{morphism of windows} $(M, M_1, \Psi_M) \rightarrow (K, K_1, \Psi_K)$  is a $W'$ linear map $f$ from $M$ to $K$, such that $f(M_1) \subset K_1$ and $f \circ \Psi_M = \Psi_k \circ f$.
\end{enumerate}
\end{definition}

Note that, as in ~\cite{GSpin}, if $(M, M_1, \Psi)$ is a Dieudonn\'e $W'$-window over $k'$, then $M_1 = \Psi^{-1}(pM)$. By ~\cite{Windows}, the category of Dieudonn\'e $W'$-windows over $k'$ is equivalent to the category of $p$-divisible groups over $k'$.

\begin{proof} (of Proposition ~\ref{dd1}) If $k = k'$, we will use the (covariant) theory of Dieudonn\'e modules: given a $k$-point $\rho: G \rightarrow \mathbb{G}$ of $\mathcal{N}$, there is some $n \in \Z$ such that $p^n \rho: G \rightarrow \mathbb{G}$ is an isogeny. If $M$ is the Dieudonn\'e module of $G$ and $\mathbb{M}$ is the Dieudonn\'e module of $\G$, we have an induced morphism of Dieudonn\'e modules:
$$\widetilde{p^n \rho}: M \rightarrow \mathbb{M}.$$
We'll use the notation $\rho(M)$ throughout this paper for $p^{-n}  \widetilde{p^n \rho} (  M) \subset \mathbb{N}$. This defines a map:
$$\mathcal{N}(k) \rightarrow  \mathcal{M}(k)$$
$$(\rho: G \rightarrow \mathbb{G} ) \mapsto  \rho(M).$$

This is a bijection using the equivalence of categories between Dieudonn\'e modules over $W$ and $p$-divisible groups over $k$.

For a general field $k'$, we will use Zink's theory of windows: because our basepoint p-dvisibile group $\G$ is defined over $k$, it has a Dieudonn\'e module $\M$ with Frobenius $F$ and isocrystal $\N$. The  Dieudonn\'e  $W'$-window associated to $\G_{k'}$ is simply the free $W$'-module $\M'$, with submodule $\M_1' = F^{-1}(p \M')$ and $\sigma$-semilinear map $F: \M' \rightarrow \M'$. 

Consider any $k'$-point $\rho: G \rightarrow \mathbb{G}_{k'}$ of $\mathcal{N}$, with associated window $(M, M_1, \Psi)$. Just as before, there is some $n \in \Z$ such that $p^n \rho$ is an isogeny, which will induce a morphism of windows $\widetilde{p^n \rho}: (M, M_1, \Psi) \rightarrow (\M', \M'_1, F)$. Then $\rho(M) = p^{-n} \widetilde{p^n \rho} (M)$ will be a Dieudonn\'e lattice containing $p^{-n} \widetilde{p^n \rho} (M_1) = (\rho(M))_1$.  Using the equivalence of categories between  $p$-divisible groups over $k'$ and  Dieudonn\'e $W'$-windows over $k'$, the map sending $(M, M_1, \Psi)$ to the Dieudonn\'e lattice $\rho(M)$ is a bijection.


\end{proof}

\subsection{Reduction to Height 0} The Rapoport-Zink space $\mathcal{N}$ decomposes into an infinite disjoint union based on the height of the quasi-isogeny, and the whole space may be reconstructed from the height-zero component. We will now explain this decomposition.

\begin{definition}
Let $f: G \rightarrow G'$ be an isogeny of $p$-divisible groups over a base scheme $S$. Then the kernel of $f$ is a finite group scheme of rank a power of $p$. If the rank is $p^i$ for some constant $i \in \Z$, then $i$ is called \emph{the height of the isogeny} f, and we write:
$$\mathrm{ht}(f) = i.$$
Let $\rho: G \rightarrow G'$ be a quasi-isogeny of $p$-divisible groups. Given $n \in \Z$ such that $p^n \rho: G \rightarrow G'$ is an isogeny, we define \emph{the height of the quasi-isogeny} $\rho$ to be:
$$\mathrm{ht}(\rho) = \mathrm{ht}(p^n \rho) - \mathrm{ht}(p^n).$$
\end{definition}

We will use the following notation: for any integer $i \in \Z$, let $\mathcal{N}_i$ be  the functor from $\Nilp_W$ to Sets given by:
$$\mathcal{N}_{i}(S) = \{ (G, \rho) \in \mathcal{N}(S)  \ | \  \rho \text{ is of height } i \}.$$

\begin{proposition} Let $\mathcal{N}_{red}$ and $\mathcal{N}_{i,red}$ be the underlying reduced $k$-schemes of $\mathcal{N}$ and $\mathcal{N}_i$, respectively. Then, 
$$\mathcal{N}_{red} = \bigsqcup_{i \in \Z} \mathcal{N}_{i,red}.$$
And the $ \mathcal{N}_{i,red}$ are precisely the connected components of $\mathcal{N}_{red}$
\end{proposition}

\begin{proof}

This follows from ~\cite{Viehmann}, Theorem 3.1.

\end{proof}

There is also a natural notion of the height of a Dieudonn\'e lattice: 
 \begin{definition}
Let $A \subset \mathbb{N}'$ be a $W'$-lattice. As $A$ is a free module over $W'$ of rank 4, $\bigwedge_{W'}^4 A$ is a free module over $W'$ of rank 1. Define the \emph{height of } $A$ to be the integer $i$ such that:
$$\bigwedge_{W'}^4 A = p^i \bigwedge_{W'}^4 \mathbb{M}'$$
as free $W'$-modules of rank 1. Such an integer $i$ exists, as $W'$ is a discrete valuation ring with uniformizer $p$.
 \end{definition}
 
 For any field $k'$ over $k$, we will use the notation: 
$$\mathcal{M}_i(k') = \{  A \in \mathcal{M}(k') \ | \ A   \text{ is of height } i \}.$$

These definitions are compatible in the following sense: if $(G, \rho) \in \mathcal{N}(k')$ and $\rho$ is a quasi-isogeny of height $i$, then the  corresponding Dieudonn\'e lattice $\rho(M) \subset \N'$ is a $W'$-lattice of height $i$. (See, for example, ~\cite{Viehmann}.) With this observation, the bijection in Proposition ~\ref{dd1} respects height, and so restricts to a bijection:
$$\mathcal{N}_{i,red}(k') \rightarrow  \mathcal{M}_{i}(k')$$
for any $i \in \Z$.

\begin{proposition}\label{allheights}
 For any $i \in \Z$,
$$\mathcal{N}_i \cong \mathcal{N}_0.$$
\end{proposition}

\begin{proof}
Define a automorphisms of the basepoint isocrystal $\N$ by: 
$$\varphi: e_1\mapsto e_2 \quad e_2 \mapsto pe_1 \quad e_3 \mapsto e_3 \quad e_4 \mapsto e_4$$
$$\psi: e_1 \mapsto e_2  \quad e_2 \mapsto pe_1 \quad e_3 \mapsto pe_3 \quad e_4 \mapsto pe_4$$
These commute with the Frobenius operator $F$ and stabilize $\mathbb{M}$, and so induce isogenies $\varphi$ and $\psi$ of $\mathbb{G}$. Note that $\varphi$ is of height 1.

Then for any $i \in \Z$ we have a morphism of functors on $\Nilp_W$, given by:
$$\mathcal{N}_i (S) \rightarrow \mathcal{N}_{i+1} (S) $$
$$(G, \rho) \mapsto (G, \varphi_{S_0} \circ \rho) $$
with inverse:
$$\mathcal{N}_{i+1} (S) \rightarrow \mathcal{N}_{i} (S) $$
$$(G, \rho) \mapsto (G, p^{-1} \psi_{S_0} \circ \rho) $$

This defines an isomorphism $\mathcal{N}_i \cong \mathcal{N}_{i+1}$, so every $\mathcal{N}_i$ is isomorphic to $\mathcal{N}_0$.

\end{proof}


 \section{From Dieudonn\'e Lattices to Very Special Lattices}
 
 In the previous section, we reduced the study of the $\GL_4$ Rapoport-Zink space $\mathcal{N}$ to the study of its height-zero component $\mathcal{N}_0$, and we showed that the $k'$-points of $\mathcal{N}_{0,red}$ can be considered as the collection $\mathcal{M}_0(k')$ of Dieudonn\'e lattices of height 0 inside our basepoint isocrystal $\N'$. Unfortunately, this does not obviously have the structure of an algebraic variety. The eventual goal is to realize a  subset of $\mathcal{M}_0(k')$ as the $k'$-points of a certain subvariety of an orthogonal Grassmanian. In this section we will begin work towards that goal by introducing a quadratic space and converting our Dieudonn\'e lattices to a collection of very special lattices in that quadratic space.

 \subsection{A Quadratic Space}

 Define:
$$V : = \bigwedge^2_{W_{\Q} } \mathbb{N}.$$

Choose some $\omega \in \bigwedge_{W}^4 \M$ such that $\bigwedge_{W}^4 \M = W \omega$ as free $W$-modules of rank 1. This choice allows us to define a $W_{\Q}$-valued symmetric bilinear form $[ \cdot  , \cdot  ]$ on $V$ by the equation:
$$[ x, y ] \omega = x \wedge y$$
for any $x,y \in V$. Note that $(V, [\cdot,\cdot])$ a nondegenerate quadratic space of dimension 6 over $W_{\Q}$.

To simplify later computations, we will study $V$ with respect to a particular basis. We have already chosen a basis $\{e_i\}_1^4$ for $\N$ as a vector space over $W_{\Q}$. Then, the vectors:
$$x_1 = e_1 \wedge e_2  \quad \quad x_2 = e_3 \wedge e_4 $$
$$x_3 = e_1 \wedge e_3 \quad \quad x_4 = e_2 \wedge e_4  $$
$$x_5 = e_1 \wedge e_4  \quad \quad x_6 = e_2 \wedge e_3 $$
 form a basis of $V$.

Note that $\omega = \alpha p^r e_1 \wedge e_2 \wedge e_3 \wedge e_4$ for some $r \in \Z$ and $\alpha \in W^\times$. Then, with respect to the basis above, we have:

  $$(a_{j,k}) = ([x_j, x_k]) = \alpha^{-1}p^{-r}
 \begin{pmatrix}
  0 & 1 &  & &  & \\
  1 & 0 & & & &\\
  & & 0 & -1 & &  \\
  & & -1 & 0 & & \\
 & & & & 0 & 1 \\
 & & & & 1 & 0 \\
 \end{pmatrix}.$$

 But $V$ has additional structure beyond being a $W_{\Q}$-vector space. Using the Frobenius operator $F$ on $\N$, we can define an operator $\Phi$ on  $V$ by:
$$\Phi(a \wedge b) = p^{-1} (F a) \wedge (Fb)$$
for any $a,b \in V$.

We will also record the action of $\Phi$ with respect to our basis chosen above. Note that $\Phi$ is $\sigma$-semilinear, so $\Phi \circ \sigma^{-1}$ is a $W_{\Q}$-linear map. Then:

$$\Phi \circ \sigma^{-1} = 
 \begin{pmatrix}
  -1 & 0 & & &  & \\
  0 & -1 & & & &\\
  & & 0 & p & &  \\
  & & p^{-1} & 0 & & \\
 & & & & 0 & 1 \\
 & & & & 1 & 0 \\
 \end{pmatrix}.$$
 
 \begin{proposition}
 $(V, \Phi)$ is a slope-zero isocrystal.
 \end{proposition}
 
 \begin{proof}
 It is enough to produce a basis on which $\Phi$ acts trivially. 
 Let $\Delta \in \Z_p^\times$ be a nonsquare and let $u \in W$ be a square root of $\Delta$, so $u^\sigma = - u$. Then, 
  $$\{ u (x_1 + x_2), u (x_1 -  x_2), u(p x_3 - x_4), px_3 + x_4, x_5 + x_6, u(x_5 - x_6) \}$$
  is a basis on which $\Phi$ acts trivially.

 \end{proof}

Let $V'$ be $V$ extended $W_{\Q}'$-linearly to a quadratic space over $W'_{\Q}$.  Just as the Frobenius $F$ may be extended to an operator on $\N'$, also denoted $F$, that is not necessarily surjective, the operator $\Phi$  has a unique $\sigma$-semilinear extension to $V'$, which will also be denoted $\Phi$. This operator is also not necessarily surjective. If $L \subset V$ is a $W'$-lattice, we will use the notation $\overline{\Phi}(L)$ for the $W'$-module generated by $\Phi(L)$.

\subsection{Hodge Star Operators and Special Endomorphisms} 
Our current objective is to relate (a subvariety of) the $\GL_4$ Rapoport-Zink space to (a subvariety of) an orthogonal Grassmanian. To show that our morphism (to be defined in Section ~\ref{Section5}) is algebraically defined, it will be necessary to realize our vector space $V$ as a collection of special endomorphisms of a certain isocrystal. To show that this morphism is an isomorphism, it will be necessary to relate $\GL(\N)$ to $\mathrm{GSpin}(V)$. Both of these results will rely on the \emph{Hodge Star Operators}, defined in this section.


Recall that $V = \bigwedge_{W_{\Q}}^2 \N$, and we have defined a pairing on $V$ by:
$$[\cdot, \cdot] : \bigwedge_{W_{\Q}}^2 \N \times \bigwedge_{W_{\Q}}^2 \N\rightarrow W_{\Q}$$
$$[x,y] \omega = x \wedge y.$$

We will use the notation $\N^* = \Hom_{W_{\Q}}(\N, W_{\Q})$. The evaluation map $\N \times \N^* \rightarrow W_{\Q}$ given by $\mathrm{ev}(a,f) = f(a)$ induces a map on every exterior power:
$$\bigwedge_{W_{\Q}}^l \N \times \bigwedge_{W_{\Q}}^l \N^{*} \rightarrow W_{\Q}$$
$$\{ a_1 \wedge a_2 \cdots \wedge a_l, f_1 \wedge f_2 \cdots \wedge f_l \} = \sum_{\pi \in S_{l} } \mathrm{sgn}(\pi) \prod_{i=1}^l \mathrm{ev}(a_i, f_{\pi(i)} ).$$
We will use this especially in the case that $l = 2$, in which case:
$$\{ a \wedge b, f \wedge g \} = f(a)g(b) - f(b)g(a).$$

 There is a unique element $\omega_1 \in \bigwedge_{W_{\Q}}^4 \N^*$ such that $\{ \omega, \omega_1 \} = 1$. Using this, define a pairing $[\cdot, \cdot]_1$ on $\bigwedge_{W_{\Q}}^2 \N^*$ by:
$$[ \cdot, \cdot]_1 : \bigwedge_{W_{\Q}}^2 \N^* \times \bigwedge_{W_{\Q}}^2 \N^* \rightarrow W_{\Q}$$
$$[t,s]_1 \omega_1 = t \wedge s.$$

\begin{proposition}
For every $x \in \bigwedge_{W_{\Q}}^2 \N$, there exists a unique $x^\star \in \bigwedge_{W_{\Q}}^2 \N^*$ such that, for all $t \in \bigwedge_{W_{\Q}}^2 \N^*$:
$$\{x, t \} = [x^\star, t]_1.$$

For every $t \in \bigwedge_{W_{\Q}}^2 \N^*$, there exists a unique $t^\star \in \bigwedge_{W_{\Q}}^2 \N$ such that, for all $x \in \bigwedge_{W_{\Q}}^2 \N$:
$$\{x, t \} = [x, t^\star].$$

The maps $\bigwedge_{W_{\Q}}^2 \N \rightarrow \bigwedge_{W_{\Q}}^2 \N^*$, $x \mapsto x^\star$ and $\bigwedge_{W_{\Q}}^2 \N^* \rightarrow \bigwedge_{W_{\Q}}^2 \N$, $t \mapsto t^\star$ will both be referred to as the \emph{Hodge star operator}.
\end{proposition}

\begin{proof}

The key point is that, because the pairings $\{ \cdot, \cdot\}$, $[\cdot, \cdot]$, and $[\cdot, \cdot]_1$ described above are nondegenerate, we have isomorphisms:
$$\bigwedge_{W_{\Q}}^2 \N \cong ( \bigwedge_{W_{\Q}}^2 \N^* )^* \quad \text{and} \quad  \bigwedge_{W_{\Q}}^2 \N^* \cong ( \bigwedge_{W_{\Q}}^2 \N^* )^*$$
$$x \mapsto \{x, \cdot \} \quad \quad \quad\quad \quad \quad\quad \quad \quad s \mapsto [s, \cdot ]_1 $$

So, for every $x \in \bigwedge_{W_{\Q}}^2 \N$, there exists a unique $x^\star \in \bigwedge_{W_{\Q}}^2 \N^*$ such that:
$$\{x, t \} = [x^\star, t]_1$$
for all $t \in \bigwedge_{W_{\Q}}^2 \N^*$.

Similarly, we have isomorphisms:
$$\bigwedge_{W_{\Q}}^2 \N^* \cong ( \bigwedge_{W_{\Q}}^2 \N)^* \quad \text{and} \quad  \bigwedge_{W_{\Q}}^2 \N\cong ( \bigwedge_{W_{\Q}}^2 \N^* )^*$$
$$t \mapsto \{\cdot , t \} \quad \quad \quad\quad \quad \quad\quad \quad \quad y \mapsto [\cdot , y ] $$

So, for every $t \in \bigwedge_{W_{\Q}}^2 \N^*$, there exists a unique $t^\star \in \bigwedge_{W_{\Q}}^2 \N$ such that:
$$\{x, t \} = [x, t^\star]$$
for all $x \in \bigwedge_{W_{\Q}}^2 \N$.

\end{proof}

\begin{proposition}\label{Pairingiscomp}
The following equalities hold by construction:
$$[x^\star, t]_1 = \{x, t\} = [x, t^\star] \quad \text{for all } x \in \bigwedge_{W_{\Q}}^2 \N, \ \ t \in \bigwedge_{W_{\Q}}^2 \N^*.$$
The  following additional identities also hold:
$$(x^\star)^\star = x \quad \text{for all } x \in \bigwedge_{W_{\Q}}^2 \N\quad \quad {and} \quad \quad (t^\star)^\star = t \quad  \text{for all } t \in \bigwedge_{W_{\Q}}^2 \N^*.$$
Further, for any field $k'$ over $k$, if the pairings $[\cdot, \cdot]$, $[\cdot, \cdot]_1$, and $\{\cdot, \cdot\}$ are all extended $W'$-linearly, and the Hodge star operator is extended $W'$-linearly, the above equalities and identities continue to hold. 
\end{proposition}

\begin{proof}
The proof is computational. Note that $\omega = w e_1 \wedge e_2 \wedge e_3 \wedge e_4$ for some $w \in W_\Q$, and if $\{f_i\}_1^4$ is the dual basis to $\{e_i\}_1^4$, $\omega_1 = w^{-1} f_1 \wedge f_2 \wedge f_3 \wedge f_4$. We have the following basis of $\bigwedge_{W_{\Q}}^2 \N$:
$$x_1 = e_1 \wedge e_2  \quad \quad x_3 = e_1 \wedge e_3 \quad \quad   x_5 = e_1 \wedge e_4 $$
$$   x_2 = e_3 \wedge e_4 \quad \quad x_4 = e_2 \wedge e_4  \quad \quad x_6 = e_2 \wedge e_3. $$

And the following basis for $\bigwedge_{W_{\Q}}^2 \N^*$:
$$t_1 = f_1 \wedge f_2  \quad \quad   t_3 = f_1 \wedge f_3 \quad \quad   t_5 = f_1 \wedge f_4$$
$$t_2 = f_3 \wedge f_4 \quad \quad t_4 = f_2 \wedge f_4 \quad \quad t_6 = f_2 \wedge f_3.  $$

One may check by hand that with the pairings defined in terms of this choice of $\omega$ and $\omega_1$, and in terms of the bases defined above, that the following holds:
\begin{equation*}
  \begin{split}
 x_1^\star &= w^{-1} t_2 \\
 x_2^\star &= w^{-1} t_1 \\
 x_3^\star &= -w^{-1} t_4 \\
  \end{split}
\quad  \quad
  \begin{split}
 t_2^\star &= w x_1 \\
 t_1^\star &= w x_2 \\
 t_4^\star &= - w x_3 \\
  \end{split}
  \quad \quad
  \begin{split}
 x_4^\star &= -w^{-1} t_3 \\
 x_5^\star &= w^{-1} t_6 \\
 x_6^\star &=w^{-1}  t_5 \\
  \end{split}
  \quad \quad
  \begin{split}
   t_3^\star &= -w x_4 \\
 t_6^\star &= w x_5 \\
 t_5^\star &= w x_6 \\
  \end{split}
\end{equation*}
From which we have that applying the Hodge star operator twice is the identity.

Finally, note that as all the definitions and arguments above were linear, for any field $k'$ over $k$, the equalities and identities in the proposition continue to hold after extending $W'_{\Q}$-linearly.

\end{proof}

We will now show that $V = \bigwedge_{W_{\Q}}^2 \N$ can be considered as a collection of endomorphisms of a certain isocrystal, using the Hodge star operators defined above. There are injections:
$$\bigwedge_{W_{\Q}}^2 \N \rightarrow \Hom(\N^*, \N) \quad \quad \quad \quad \quad \quad \quad \bigwedge_{W_{\Q}}^2 \N^* \rightarrow \Hom (\N, \N^*)$$
$$a \wedge b \mapsto a \wedge b: (f \mapsto f(a)b - f(b)a) \quad \quad f \wedge g \mapsto f \wedge g: (c \mapsto g(c)f - f(c)g).$$

 Note that $\N^*$ is also an isocrystal, with $F$ and $V$ operators induced by those on $\N$: for any $\varphi \in \N^*$ and $n \in \N$,
$$(F \varphi)(n) = \varphi(V n)^\sigma \quad \quad \text{and} \quad \quad (V \varphi)(n) = \varphi(Fn)^{\sigma^{-1} }.$$
Let $\underline{\N} = \N \times \N^*$. Then $\underline{\N}$ is also an isocrystal, with the $F$ and $V$ operators defined on the two factors.

Finally, we can define:
$$V = \bigwedge_{W_{\Q}}^2 \N\rightarrow \End(\underline{\N} )$$
$$x \mapsto \tilde{x}$$
$$\tilde{x}: \underline{\N}  =  \N \times \N^* \xrightarrow{(x^\star, x)} \N^* \times \N \cong \underline{\N} .$$

\begin{definition}
Define the \emph{special endomorphisms} of $\underline{\N}$ to be the image of $\bigwedge_{W_{\Q}}^2 \N$ in $\End(\underline{\N} )$ under the map above. This collection will be denoted $\End_{\star}(\underline{\N})$.
\end{definition}

We will not use the following observation, but it is worth noting that the special endomorphisms of $\underline{\N}$ have a simple description. An element $f \in \End(\underline{\N})$ is a special endomorphism if and only if the following conditions hold:
\begin{enumerate}
\item{The endomorphism $f$ takes $\N$ to $\N^*$ and $\N^*$ to $\N$ , so can be considered as a pair:
$$(f_1, f_2): \N \times \N^* \rightarrow \N^* \times \N$$}
\item{Under the canonical identification $\N \cong (\N^*)^*$, $f_1^* = -f_1$ and $f_2^* = -f_2$}
\item{The endomorphism $f_2$ is naturally induced by $f_1$ via the Hodge star operator.}
\end{enumerate}
To understand the third condition, note that the second condition implies that $f_1$ is the endomorphism induced by some $t \in \bigwedge_{W_{\Q}}^2 \N^*$. The third condition is that $f_2$ must be induced by $t^\star \in \bigwedge_{W_{\Q}}^2 \N$.

\begin{proposition}\label{compositionform}
For any $x,y \in V$,
$$[x,y] =  \widetilde{x} \circ \widetilde{y} + \widetilde{y} \circ \widetilde{x}.$$

So, the map:
$$V = \bigwedge_{W_{\Q}}^2 \N\rightarrow \End_{\star}(\underline{\N} )$$
$$x \mapsto \tilde{x}$$
defines an isomorphism between the quadratic spaces $(V, [\cdot, \cdot])$ and the collection of special endomorphisms $\End_{\star}(\underline{\N})$ with the natural composition form $[f,g] = f \circ g + g \circ f$. Further, for any field $k'$ over $k$, these properties continue to hold after extending $V$ and $\End_{\star}(\underline{\N})$ both $W_{\Q}'$-linearly.
\end{proposition}

\begin{proof}
As in the proof of Proposition ~\ref{Pairingiscomp}, note that $\omega = w e_1 \wedge e_2 \wedge e_3 \wedge e_4$ for some $w \in W_{\Q}$. In that proof, we observed that, for the bases $\{x_i\}_1^6$ of $ \bigwedge_{W_{\Q}}^2 \N$ and $\{t_i\}_1^6$ of $\bigwedge_{W_{\Q}}^2 \N^*$ :

\begin{equation*}
  \begin{split}
 x_1^\star &= w^{-1} t_2 \\
 x_2^\star &= w^{-1} t_1 \\
 x_3^\star &= -w^{-1} t_4 \\
  \end{split}
\quad  \quad
  \begin{split}
 t_2^\star &= w x_1 \\
 t_1^\star &= w x_2 \\
 t_4^\star &= - w x_3 \\
  \end{split}
  \quad \quad
  \begin{split}
 x_4^\star &= -w^{-1} t_3 \\
 x_5^\star &= w^{-1} t_6 \\
 x_6^\star &=w^{-1}  t_5 \\
  \end{split}
  \quad \quad
  \begin{split}
   t_3^\star &= -w x_4 \\
 t_6^\star &= w x_5 \\
 t_5^\star &= w x_6 \\
  \end{split}
\end{equation*}

From which one may  check by hand that:
$$[x_i, x_j] = \widetilde{x_i} \circ \widetilde{x_j} + \widetilde{x_j} \circ \widetilde{x_i}$$
for any $1 \leq i,j \leq 6$, which shows that $[x,y] =  \widetilde{x} \circ \widetilde{y} + \widetilde{y} \circ \widetilde{x}.$

Note that the map  
$$V = \bigwedge_{W_{\Q}}^2 \N\rightarrow \End_{\star}(\underline{\N} )$$
$$x \mapsto \tilde{x}$$
is clearly injective, and is surjective by definition, so defines an isomorphism between the quadratic space $(V, [\cdot, \cdot])$ and the collection of special endomorphisms of $\underline{\N}$ with the natural composition form. As all of the above arguments and definitions were linear, these properties continue to hold after extending $V$ and $\End_{\star}(\underline{\N})$ both $W_{\Q}'$-linearly.
\end{proof}

\subsection{Exceptional Isomorphisms}\label{excepsect}

Our objective is to convert our collection of Dieudonn\'e lattices in $\N'$ to a collection of ``very special" lattices in the quadratic space $V'$. As discussed previously, it is enough to study the Dieudonn\'e lattices of height 0. These in particular are height-0 $W'$ lattices in $\N'$, which naturally have an action of $\SL_4$. The very special lattices we will be studying in $V'$ will in particular be self-dual lattices, which naturally have an action of $\SO(V)$. In order to convert our Dieudonn\'e lattices into very special lattices, we will have to understand the relation between these groups. This will be accomplished by viewing $ \GL(\N) \times W_{\Q}^\times $ as a unitary group, and by using an exceptional isomorphism between a unitary group and a spin group.

Recall that $\underline{\N} = \N \times \N^*$, which is an isocrystal with the $F$ and $V$ operators acting diagonally on the two factors. We will define a polarization of $\underline{\N}$. If $\N \rightarrow (\N^*)^*$, $a \mapsto e_a$ is the canonical evaluation map, define:
$$\boldsymbol{\lambda}: \underline{\N} = \N \times \N^* \rightarrow \N^* \times (\N^*)^* \cong \underline{\N}^* $$
$$\boldsymbol{\lambda}(a,f) = (f, -e_a).$$

This can be considered instead as an alternating form $[ \cdot, \cdot]_{\boldsymbol{\lambda}}$ on $\underline{\N}$, where:
$$[  n_1 , n_2 ]_{\boldsymbol{\lambda}} = \boldsymbol{\lambda}(n_2)(n_1).$$

When considering an element $n_i \in \underline{\N}$ as a pair $(a_i, f_i) \in \N \times \N^*$, this alternating pairing has a particularly nice description:
$$[ (a_1, f_1), (a_2, f_2) ] _{\boldsymbol{\lambda} } = f_2(a_1) - f_1(a_2).$$

We can also define an action of  $W_{\Q} \times W_{\Q}$ on $\underline{\N} = \N \times \N^*$ by acting componentwise. Now, we may define the unitary similitude group:
$$\GU(\underline{\N}) = \{ g \in \Aut_{W_{\Q} \times W_{\Q} }(\underline{\N}) \ | \ g^*(\boldsymbol{\lambda} ) = \nu(g) \boldsymbol{\lambda} \ \text{ for some } \ \nu(g) \in W_{\Q}^\times \}.$$


Where the condition $g^*(\boldsymbol{\lambda} ) = \nu(g) \boldsymbol{\lambda}$ may be thought of in terms of either the polarization $\boldsymbol{\lambda}$ or the corresponding alternating form $[ \cdot, \cdot ] _{\boldsymbol{\lambda} }$.  This group may be defined functorially in such a way  that the group described above would be the $W_{\Q}$-points.  One would expect there to be an imaginary  quadratic field $E$ in the description of a unitary group: taking $E$ over $\Q$ to be any imaginary quadratic field in which the prime $p$ is split will produce this definition.

This group has an important subgroup: following the notation of ~\cite{GU22}, set:
$$\GU^0(\underline{\N}) = \{ g \in \GU(\underline{\N})  \ | \ (\nu(g))^2 = \det(g) \}.$$
In this definition, note that $g \in  \GU(\underline{\N})$ is $W_{\Q} \times W_{\Q}$-linear, so $\det(g) \in W_{\Q} \times W_{\Q}$. The equality $ (\nu(g))^2 = \det(g)$ takes place by identifying $\nu(g)$ with $(\nu(g), \nu(g) )$ in $W_{\Q} \times W_{\Q}$.

The key point in what follows is that much of an element of $\GU(\underline{\N})$ is determined simply by its action on one of the two factors. Precisely:

\begin{lemma}\label{GUislikeGL}
There is an isomorphism:
$$\GU(\underline{\N}) \xrightarrow{\sim} \GL(\N) \times W^\times_{\Q}$$
$$g \mapsto (g|_{\N}, \nu(g) ).$$

This isomorphism identifies $\GU^0(\underline{\N})$ with the subgroup 
$$H = \{ (h, c) \in \GL(\N) \times W^{\times}_{\Q} \ | \ \det(h) = c^2 \}$$
of $\GL(\N) \times W^\times_{\Q}$.
\end{lemma}

\begin{proof}
The inverse of the map described above is:

$$ \GL(\N) \times W^\times_{\Q} \rightarrow \GU(\underline{\N}) = \GU(\N \times \N^*)$$
$$(h, c) \mapsto (h, c(h^{-1})^* ).$$

The key observation is that any $g \in \GU(\underline{\N})$ must take the component $\N$ to itself and the component $\N^*$ to itself, due to the $W_{\Q} \times W_{\Q}$-linearity. Futher, the condition that $g^*(\boldsymbol{\lambda}) = \nu(g) \boldsymbol{\lambda}$ implies that $g$ must be of the form $(g|_{\N}, \nu(g) (g|^{-1}_{\N})^*)$.

To see that this isomorphism identifies the subgroups $H$ and $\GU^0(\underline{\N})$, let $(h, c) \in \GL_{\N} \times W_{\Q}^\times$ map to $g \in \GU(\underline{\N})$, so that when viewed on the components $\N$ and $\N^*$ of $\underline{\N}$:
$$g = (h, c(h^{-1})^*) \quad \quad \text{and} \quad \quad \nu(g) = c.$$

Then, the determinant of $g$ as a $W_{\Q} \times W_{\Q}$ linear automorphism of $\underline{\N}$ is:
$$( \det(h) , c^4 \det(h)^{-1} ).$$

The condition that $\det(g) = \nu(g)^2$ corresponds exactly to the condition that $\det(h) = c^2$, and so the isomorphism between $\GU(\underline{\N})$ and $\GL(\N) \times W^\times_{\Q}$ identifies $\GU^0(\underline{\N})$ with $H$.

\end{proof}

The group $\GL(\N) \times W^\times_{\Q} \cong \GU(\underline{\N})$ acts on $V = \bigwedge_{W_{\Q}}^2 \N$ by:
$$(h,c) \cdot a \wedge b  = c^{-1} h(a) \wedge h(b).$$
Note that, when $(h,c) \in  H \cong \GU^0(\underline{\N})$, action by $(h,c)$ preserves the quadratic form on $V$.

The map $W_{\Q}^\times \rightarrow \GL(\underline{\N})$ has image in $\GU^0(\underline{\N})$, so defines a map $W_{\Q}^\times \rightarrow H$. Concretely,  $w \in W_{\Q}^\times$ maps to $(w\cdot I, w^2 ) \in H$. Combining these, we have a sequence (which we haven't yet shown is exact):
$$1 \rightarrow W_{\Q}^\times \rightarrow H \rightarrow \SO(V) \rightarrow 1.$$
The key point in showing this sequence is exact is the following exceptional isomorphism:

\begin{lemma}\label{GU22Gspin}
There is an isomorphism:
$$\mathrm{GU}^0(\underline{\N}) \cong \mathrm{GSpin}(V)$$
in a way compatible with the action of both groups on $V$. 
\end{lemma}

\begin{proof}
The argument is analogous to ~\cite{GU22} Proposition 2.7. It is worth mentioning that this argument relies on the Hodge star operators and the isomorphism of quadratic spaces between $(V, [\cdot,\cdot])$ and $\mathrm{End}_{\star}(\underline{\N})$ with its composition form.
\end{proof}

\begin{proposition}\label{Usefulsequence}
The sequence described above:
$$1 \rightarrow W_{\Q}^\times \rightarrow H \rightarrow \SO(V) \rightarrow 1.$$
 is exact.
\end{proposition}

\begin{proof}
By ~\cite{AoQF}, there is a short exact sequence:
$$1 \rightarrow W_{\Q}^\times \rightarrow \mathrm{GSpin}(V) \rightarrow \SO(V) \rightarrow 1.$$
Combining the above with the results of Lemmas ~\ref{GUislikeGL} and  ~\ref{GU22Gspin} yields the exact sequence:
$$1 \rightarrow W_{\Q}^\times \rightarrow H \rightarrow \SO(V) \rightarrow 1.$$

\end{proof}

 \subsection{Special and Very Special Lattices}

 The final objective of this section it to create a bijection between the Dieudonn\'e lattices of height $0$ in $\mathbb{N}'$  and a set of ``very special" lattices in $V'$. We now have the background necessary to accomplish this.

  \begin{definition}\label{defspecial}
A $W'$-lattice $L$ in the quadratic space $V'$ is called \emph{special} if the following two conditions hold:
\begin{enumerate}
\item{$L$ is self-dual with respect to the pairing $[\cdot, \cdot]$.}
\item{$(L + \overline{\Phi}(L) )/L$ has length 1.}
\end{enumerate}
Let $\mathcal{S}(k')$ denote the collection of all special lattices in $V'$.
  \end{definition}
  
  The point of the following lemma is that special lattices in $V'$ naturally come in two types. Let:
  $$\mathcal{S}^+(k') = \{ L \in \mathcal{S}(k)  \ | \ L = \frac{1}{p} \bigwedge_{W'}^2 A \text{ for some } W'\text{-lattice } A \subset \N' \}$$
  $$\text{and}$$
   $$\mathcal{S}^-(k') = \{ L \in \mathcal{S}(k) \ | \ L = \bigwedge_{W'}^2 A \text{ for some } W'\text{-lattice } A \subset \N' \}.$$ 
  
  \begin{lemma}\label{twotypesoflattices}
 For any field $k'$, the set of special lattices decompose into a disjoint union of the two types above:
 $$\mathcal{S}(k') = \mathcal{S}^+(k') \sqcup \mathcal{S}^-(k').$$
  \end{lemma}
  
  \begin{proof}
One can use the elementary divisor theorem to see that a special lattice can be of at most one of the two types.

 Now, to show that every special lattice is contained in a least one of $\mathcal{S}^+(k')$ or $\mathcal{S}^-(k')$, let $L$ be a special lattice. Note that $\bigwedge_{W'}^2 \mathbb{M}'$ is a self-dual lattice in $V'$, and since $L$ is also self-dual there is some $g \in \SO(V')$ such that $g(\bigwedge_{W'}^2 \mathbb{M}') = L$. By Proposition ~\ref{Usefulsequence}, there is some $(h,c) \in H \subset \GL(\N') \times W'^\times_{\Q}$ mapping to $g$. So:
 $$L = g(\bigwedge_{W'}^2 \mathbb{M}') = (h,c) \cdot (\bigwedge_{W'}^2 \mathbb{M}') = \frac{1}{c} \bigwedge_{W'}^2 h\mathbb{M}'$$
 as free $W'$ modules in $V'$. Then either $\ord_p(c)$ is even, in which case $L \in \mathcal{S}^-(k')$, or $\ord_p(c)$ is odd, in which case $L \in \mathcal{S}^+(k')$.
  \end{proof}

  \begin{definition}\label{defvspecial}
  A $W'$-lattice $L$ in the quadratic space $V'$ is called \emph{very special} if it is in the component $\mathcal{S}^+(k')$ of $\mathcal{S}'(k')$. Precisely, a $W'$-lattice $L$ in the quadratic space $V'$ is called \emph{very special} if  if it is meets the following three conditions:
  \begin{enumerate}
\item{$L$ is self-dual with respect to the pairing $[\cdot, \cdot]$.}
\item{$(L + \overline{\Phi}(L) )/L$ has length 1.}
\item{ $L = \frac{1}{p} \bigwedge_{W'}^2 A$ for some $W'$-lattice $A \subset \N'$. }
\end{enumerate}
\end{definition}

Our goal for this section is the following proposition:

  \begin{proposition}\label{bijection}
  There is a bijection between the set of Dieudonn\'e lattices of height 0 in the basepoint isocrystal $\N'$ and the set of very special lattices in $V' = \bigwedge_{W'_{\Q}}^2 \N'$:
$$ \mathcal{M}_{0}(k') \longleftrightarrow \mathcal{S}^+(k')$$
Given by $A \mapsto \frac{1}{p} \bigwedge_{W'}^2 A_1$.
  \end{proposition}

  The proof of this will require a few lemmas. We will frequently use the following notation: for $A \subset V'$ a $W'$ lattice, let $A^\vee$ denote the dual lattice with respect to $[\cdot,\cdot]$. This is  the collection of elements in $V'$ that pair integrally with all of $A$.

 \begin{lemma}\label{isdual}
  Given a $W'$-lattice $A \subset \mathbb{N}'$ of height $2i$ and an integer $k$, $p^k\bigwedge_{W'}^2 A$ is a self-dual lattice in $V'$ if and only if $k = -i$.   
  \end{lemma}
  
  \begin{proof}
  This may be checked using the elementary divisor theorem.

  \end{proof}

Now, to show that the map $ \mathcal{M}_{0}(k') \longrightarrow \mathcal{S}^+(k')$ given by $A \mapsto \frac{1}{p} \bigwedge_{W'}^2 A_1$ does in fact have image in $ \mathcal{S}^+(k')$:

  \begin{lemma}\label{isspecial}
  For any  Dieudonn\'e lattice $A \subset \mathbb{N}'$ of height $0$, the lattice $\frac{1}{p} \bigwedge_{W'}^2 A_1 \subset V'$ is very special.   
  \end{lemma}

  \begin{proof}

 Let $A \subset \N'$ be a Dieudonn\'e lattice of height $0$. Only because $A$ is a  \emph{Dieudonn\'e} lattice of height $0$, $A_1$ is of height 2. Then,   by Lemma ~\ref{isdual}, $ \frac{1}{p} \bigwedge_{W'}^2 A_1$  is self-dual.

Now we will show that $(\frac{1}{p} \bigwedge_{W'}^2 A_1 + \Phi( \frac{1}{p} \bigwedge_{W'}^2 A_1 ) )/ \frac{1}{p} \bigwedge_{W'}^2 A_1$ has length 1. Because $A$ is a Dieudonn\'e lattice, we have:

\begin{enumerate}
\item{ $pA  \subset A_1 \subset A$}
\item{ $A = \overline{F}(F^{-1}(A) )$}
\end{enumerate}

The second condition implies that $\dim_{k'}(A_1/pA) = 2$, and so $\dim_{k'}(A/A_1) = 2$.

Using the elementary divisor theorem, let $\{ a_j \}_1^4$ be a free $W'$-basis for $A$ such that $A_1$ has free $W'$-basis $\{ c_j a_j \}_1^4$ for some $c_j \in W'_{\Q}$. Because of the dimension conditions above, and by scaling by $W'^\times$, we may assume without loss of generality that $c_1 = c_2 = p$ and $c_3 = c_4 = 1$. With this choice of basis, it is clear that  $(\frac{1}{p} \bigwedge_{W'}^2 A_1 + \Phi( \frac{1}{p} \bigwedge_{W'}^2 A_1 ) )/ \frac{1}{p} \bigwedge_{W'}^2 A_1$  has length 1.
    
  
 
 The self-duality and the length one condition together show that $L$ is a special lattice. As $L$ is of the form $\frac{1}{p} \bigwedge_{W'}^2 A_1$, it is a very special lattice. 
  \end{proof}

  \begin{lemma}\label{injective} The map
 $$ \mathcal{M}_0(k') \longrightarrow \mathcal{S}^+(k')$$
given by  $A \mapsto \frac{1}{p} \bigwedge_{W'}^2 A_1$ is injective.
  \end{lemma}
  
  \begin{proof}
  Let $A$ and $C$ be two Dieudonn\'e lattices of height $0$ in $\mathbb{N}'$, and assume that $\frac{1}{p} \bigwedge_{W'}^2 A_1 = \frac{1}{p} \bigwedge_{W'}^2 C_1$ as special lattices in $V'$. For the sake of contradiction, assume that $A \neq C$.

  By the elementary divisor theorem, there is a free $W'$-basis $\{a_j\}_1^4$ of $A$ such that $C$ has free $W'$-basis $\{c_j a_j \}_1^4$ for some $c_j \in W'_{\Q}$. Because $A \neq C$, there must be some $c_j$ with $\ord_p(c_j) \neq 0$. Without loss of generality, we may assume $\ord_p(c_1) > 0$.
  
  Because  $\frac{1}{p} \bigwedge_{W'}^2 A_1 = \frac{1}{p} \bigwedge_{W'}^2 C_1$, we have that $\overline{\Phi}( \bigwedge_{W'}^2 F^{-1}(pA) ) = \overline{\Phi}( \bigwedge_{W'}^2 F^{-1}(pC))$, and so $\bigwedge_{W'}^2 A = \bigwedge_{W'}^2 C$.
 
 Then $\bigwedge_{W'}^2 A$ has free $W'$-basis $\{ a_j \wedge a_k \}_{j<k}$ and $\bigwedge_{W'}^2 C$ has free $W'$-basis $\{ c_j c_k a_j \wedge a_k \}_{j < k}$. Because $\bigwedge_{W'}^2 A = \bigwedge_{W'}^2 C$, $\ord_p(c_jc_k) = 0$ for any pair $j<k$. Because $\ord_p(c_1) > 0$, we must have $\ord_p(c_2)  = \ord_p(c_3)  = \ord_p(c_4) = - \ord_p(c_1) < 0$. In particular, $\sum_{j=1}^4 \ord_p(c_j) < 0$.
  
However, because $A$ and $C$ are both height $0$, we must have $\sum_{j=1}^4 \ord_p(c_j) = 0$, which is a contradiction.
  \end{proof}

\begin{lemma}\label{isDD}
Let $A \subset \mathbb{N}'$ be a lattice such that $(\frac{1}{p} \bigwedge_{W'}^2 A_1  + \overline{\Phi}( \frac{1}{p} \bigwedge_{W'}^2 A_1 ) )/  \frac{1}{p} \bigwedge_{W'}^2 A_1$ has length 1 and such that $\overline{F}(F^{-1}(A)) = A$. Then $A$ is a Dieudonn\'e lattice.
\end{lemma}

\begin{proof}

Let $A \subset \mathbb{N}'$ be a lattice such that $(\frac{1}{p} \bigwedge_{W'}^2 A_1  + \overline{\Phi}( \frac{1}{p} \bigwedge_{W'}^2 A_1 ) )/  \frac{1}{p} \bigwedge_{W'}^2 A_1$ has length 1 and such that $\overline{F}(F^{-1}(A)) = A$.  To show that $A$ is a Dieudonn\'e lattice, we only need to show that $pA \subset A_1 \subset A$. 

By the elementary divisor theorem, there is a free $W'$-basis $\{a_i\}_1^4$ of $A$ such that $A_1$ has free $W'$-basis $\{c_ia_i\}_1^4$ for some $c_i \in W'_{\Q}$, so $\bigwedge_{W'}^2 A_1$ has free basis $\{ c_j c_k  a_j \wedge a_k \}_{j<k}$.  Because  $\overline{F}(F^{-1}(A)) = A$, we have that $\overline{\Phi}( \bigwedge_{W'}^2 A_1 ) = p\bigwedge_{W'}^2 A$, which  has free $W'$-basis $\{p a_j \wedge a_k \}_{j<k}$. The key observation is that, because $(V', \Phi)$ is a slope-zero isocrystal, $\ord_p (\prod_{j<k} \frac{c_j c_j}{p } ) = 0$.

Note that $( \frac{1}{p} \bigwedge_{W'}^2 A_1  + \overline{\Phi}( \frac{1}{p} \bigwedge_{W'}^2 A_1  ) )$ has free basis $\{ \min(1, \frac{c_jc_k}{p} ) a_j \wedge a_k \}_{j<k}$ where the minimum is taken with respect to $\ord_p$, while $\frac{1}{p} \bigwedge_{W'}^2 A_1$ has free basis $\{ \frac{c_jc_k}{p} a_j \wedge a_k\}_{j<k}.$

Because$(\frac{1}{p} \bigwedge_{W'}^2 A_1  + \overline{\Phi}( \frac{1}{p} \bigwedge_{W'}^2 A_1 ) )/  \frac{1}{p} \bigwedge_{W'}^2 A_1$ has length 1, we may assume without loss of generality that:
$$\ord_p(\frac{c_1 c_2}{p}) = 1$$
$$\ord_p(\frac{c_j c_k}{p} ) \leq 0 \quad \text{for any } j < k, (j,k) \neq (1,2).$$
Combining these with the observation that $\ord_p (\prod_{j<k} \frac{c_j c_j}{p } ) = 0$ yields:
$$\ord_p(c_1) = 1 \quad \ord_p(c_2) = 1 \quad \ord_p(c_3) = 0 \quad \ord_p(c_4) = 0$$

And so we have the chain condition $pA \subset F^{-1}(pA) \subset A,$ which shows that $A$ is a Dieudonn\'e lattice.

\end{proof}

Now we are ready to prove Proposition ~\ref{bijection}:

\begin{proof} 
By Lemma ~\ref{isspecial}, there is a well-defined map:
$$ \mathcal{M}_{0}(k') \longrightarrow \mathcal{S}^+(k')$$
Given by $A \mapsto \frac{1}{p} \bigwedge_{W'}^2 A_1$. This map is injective by Lemma ~\ref{injective}. It remains to show that every very special lattice $L$ in $V'$ is of the form $\frac{1}{p} \bigwedge_{W'}^2 A_1$ for some Dieudonn\'e lattice $A$ of height $0$. So, let $L \subset V'$ be a very special lattice. 

By definition, because $L$ is a very special lattice, $L = \frac{1}{p} \bigwedge_{W'}^2 C$ for some $W'$-lattice $C \subset \mathbb{N}'$. Let $A = \overline{F}(\frac{1}{p} C)$. Note that $C =  F^{-1}(\overline{F}(C))$ (this is true for any $W'$-lattice in $\N')$, so: 
$$L = \frac{1}{p} \bigwedge_{W'}^2 C = \frac{1}{p} \bigwedge_{W'}^2 A_1.$$

It only remains to show that $A$ is a Dieudonn\'e lattice of height 0. Because $L$ is self-dual, by Lemma ~\ref{isdual}, $C$ must be of height 2.  Since $A = \overline{F}(\frac{1}{p} C)$, $A$ is of height 0.

Because $L = \frac{1}{p} \bigwedge_{W'}^2 A_1$ is self-dual,  the operator $\Phi$ is slope-0, and $[\Phi(x), \Phi(y)] = [x,y]^\sigma$ for any $x,y \in V'$, we have that
$$\overline{\Phi}(\frac{1}{p} \bigwedge_{W'}^2 A_1) = \frac{1}{p^2}\bigwedge_{W'}^2 \overline{F}(F^{-1}(pA))$$
is also self-dual in $V'$. By applying Lemma ~\ref{isdual} again, $ \overline{F}(F^{-1}(pA))$ must be of height $4$. Since $pA$ is also of height $4$ and $ \overline{F}(F^{-1}(pA)) \subset pA$, equality must hold. Therefore, $\overline{F}(F^{-1}(A)) = A$.

Now, we have that  $A$ is a $W'$-lattice of height 0, $\overline{F}(F^{-1}(A)) = A$, and
$$ (\frac{1}{p} \bigwedge_{W'}^2 A_1  + \overline{\Phi}( \frac{1}{p} \bigwedge_{W'}^2 A_1 ) )/  \frac{1}{p} \bigwedge_{W'}^2 A_1$$
has length 1. By Proposition ~\ref{isDD}, $A$ is a Dieudonn\'e lattice. 

Therefore, our very special lattice  $L$ is of the form $\frac{1}{p} \bigwedge_{W'}^2 A_1$ for some Dieudonn\'e lattice $A$ of height $0$, and so $A \mapsto \frac{1}{p} \bigwedge_{W'}^2 A_1$ defines a bijection:

 $$ \mathcal{M}_0(k') \longleftrightarrow \mathcal{S}^+(k').$$

\end{proof}

\section{Vertex Lattices and Endomorphisms}\label{Vertexsection}

Our next goal will be study certain closed subschemes $\mathcal{N}_{\Lambda, 0} \subset \mathcal{N}_{0,red}$, where specific collections of endomorphisms are integral.  This is a key step, because the subschemes  $\mathcal{N}_{\Lambda, 0}$ will be projective varieties. The indexing set for these subschemes will be the collection of \emph{vertex lattices} in the quadratic space $V^{\Phi}$, the $\Phi$-fixed vectors in $V$. So we will begin by studying the quadratic space $V^{\Phi}$ and introducing the concept of a vertex lattice.

\subsection{A Rational Quadratic Space and Vertex Lattices}

\begin{proposition}
The quadratic space $V^{\Phi}$ over $\Q_p$ has Hasse invariant -1 and determinant -1. This does not depend on the choice of $\omega$ used to define the quadratic space $V$.
\end{proposition}

\begin{proof}

 Let $\Delta \in \Z_p^\times$ be nonsquare, and let $u \in W^\times$ be a square root of $\Delta$. The vectors:
 
 $$y_1 =u (x_1 + x_2)  \quad \quad y_2 = u (x_1 -  x_2)$$
 $$y_3 = u(p x_3 - x_4) \quad \quad y_4 = px_3 + x_4$$
 $$y_5 =  x_5 + x_6  \quad \quad y_6 = u(x_5 - x_6) $$
 
 form an orthogonal basis of $V^{\Phi}$, with:

   $$\left(\alpha p^r \frac{[y_i, y_j] }{2 } \right) = 
 \begin{pmatrix}
  \Delta &  & & &  & \\
   & -\Delta & & & &\\
  & & \Delta p &  & &  \\
  & &  & -p & & \\
 & & & & 1 &  \\
 & & & &  & - \Delta \\
 \end{pmatrix}.$$

 This has Hasse invariant -1 and determinant -1.
 
 Recall that $\omega$ was chosen to be an element of $\bigwedge_{W_{\Q}}^4 \N$ such that $W \omega = \bigwedge_{W}^4 \M$ as free $W$-modules of rank 1. If $\omega'$ were another such element, we would have $\omega' = \beta \omega$ for some $\beta \in W^\times$. Let $[\cdot,\cdot]'$ be the pairing based on $\omega'$. Then $\{y_i\}_1^6$ would still be an orthogonal basis of $V^{\Phi}$, with:
 
    $$\left(\beta \alpha p^r \frac{[y_i, y_j'] }{2 } \right) = 
 \begin{pmatrix}
  \Delta &  & & &  & \\
   & -\Delta & & & &\\
  & & \Delta p &  & &  \\
  & &  & -p & & \\
 & & & & 1 &  \\
 & & & &  & - \Delta \\
 \end{pmatrix}.$$
 
 This still has Hasse invariant -1 and determinant -1.
 
 By ~\cite{AoQF}, there is a unique  nondegenerate anisotropic quadratic space of dimension 4 over $\Q_p$. Let $W$ be this quadratic space, and let $\mathbb{H}$ be the hyperbolic plane over $\Q_p$. Then,
 
 $$V^{\Phi} \cong  \mathbb{H} \oplus W.$$

\end{proof}

\begin{definition}\label{ignoreme}
A \emph{vertex lattice} is a $\Z_p$-lattice $\Lambda \subset V^{\Phi}$ such that:
$$p \Lambda \subset \Lambda^\vee \subset \Lambda.$$
The \emph{type} of $\Lambda$ is $t_{\Lambda} = \dim_{\F_p}(\Lambda/ \Lambda^\vee)$.
\end{definition}

\begin{proposition}
The type of a vertex lattice is either 2 or 4 (and both occur).
\end{proposition}

\begin{proof}
This follows from ~\cite{GSpin} Proposition 5.1.2 (see the introduction for the definition of $t_{max}$, which in this case is 4). Proposition 5.1.2 includes a proof that vertex lattices of type 4 do occur. The fact that vertex lattices of type 2 do occur can be seen, for example, by Proposition ~\ref{vollaardprop} below.
\end{proof}

The statement of the following proposition is identical to that of Proposition 2.19 in ~\cite{GU22}. The proof follows from ~\cite{RTW}, Proposition 4.1. Additional details may be found in ~\cite{V} Lemma 2.1.

\begin{proposition}\label{vollaardprop}
 Let $L \subset V$ be a special lattice, and define:
$$L^{(r)} = L + \Phi(L) + \cdots + \Phi^r(L).$$
There is an integer $d \in \{1,2\}$ such that:
$$L = L^{(0)} \subsetneq L^{(1)} \subsetneq \cdots \subsetneq L^{(r) } = L^{(r+1)}.$$
For each $L^{(r)} \subsetneq L^{(r+1})$ with $0 \leq r < d$ the quotient $L^{(r+1)}/L^{(r)}$ is annihilated by $p$ and satisfies $\dim_k ( L^{(r+1)}/L^{(r)} ) = 1$. Moreover,
$$\Lambda_{L} = \{x \in L^{(d)} \ | \ \Phi(x) = x \}$$
is a vertex lattice of type $2d$ and satisfies $\Lambda_{L}^\vee = \{x \in L \ | \  \Phi(x) = x \}$.
\end{proposition}

Define:
\begin{align*}
\mathcal{S}_{\Lambda}(k') &= \{ \text{Special lattices }  L \subset V'  \text{ such that }  \Lambda^\vee \subset L\} \\
&=  \{ \text{Special lattices }  L \subset V' \text{ such that }  \Lambda^\vee \subset \overline{\Phi}(L) \}. \\
\end{align*} 

The above two descriptions are equivalent because $\Lambda^\vee \subset V^\Phi$ and the lattices $L$ are \emph{special} lattices. Similarly define:
\begin{align*}
\mathcal{S}^+_{\Lambda}(k') &= \{ \text{Very special lattices }  L \subset V'  \text{ such that }  \Lambda^\vee \subset L\} \\
&=  \{ \text{Very special lattices }  L \subset V' \text{ such that }  \Lambda^\vee \subset \overline{\Phi}(L) \}. \\
\end{align*}

Note that, by the above Proposition ~\ref{vollaardprop}, every special lattice in $V$ is contained in some $S_{\Lambda}(k)$, and every very special lattice is contained in some $S^+_{\Lambda}(k)$.

Recall that we have bijections between the $k'$-points of the height-zero component of the $\GL_4$ Rapoport-Zink space, the set of Dieudonn\'e lattices of height 0 in $\N'$, and the set of very special lattices in $V'$:
$$\mathcal{N}_{0,red}(k') \longleftrightarrow \mathcal{M}_0(k') \longleftrightarrow \mathcal{S}^+(k').$$

We will show there is a natural way to define a closed subscheme $\mathcal{N}_{\Lambda, 0}$ of $\mathcal{N}_{0,red}$ so that the above bijection restricts to:

$$\mathcal{N}_{\Lambda, 0}(k') \longleftrightarrow \mathcal{S}^+_{\Lambda} (k').$$

\subsection{The Endomorphisms that Stabilize a Lattice}\label{Lattices Stabilized by Ends} The current objective is to define a closed subscheme $\mathcal{N}_{\Lambda,0}$ of $\mathcal{N}_{0,red}$ such that  $\mathcal{N}_{\Lambda,0}(k')$ that will be in bijection with $\mathcal{S}_{\Lambda}^+(k')$. In this section, we will start by defining a subset  $\mathcal{M}_{\Lambda,0}(k') \subset \mathcal{M}_{0}(k')$ that is in bijection with $\mathcal{S}_{\Lambda}^+(k')$. For $A \subset \N'$ a $W'$-lattice, we will use the notation:
$$A^* = \Hom(A, W')$$
or, equivalently, that $A^*$ is the lattice in $\N'^* = \Hom(\N', W'_{Q})$ of homomorphisms integrally valued on $A$. This should not be confused with the notation $L^\vee$: if $L \subset V'$ is a lattice, $L^\vee \subset V'$ is the lattice of elements of $V'$ that pair integrally (with respect to $[\cdot, \cdot]$) with all elements of $L$.

\begin{definition}
For any vertex lattice $\Lambda \subset V^{\Phi}$, and any field $k'$ over $k$, define:
$$\mathcal{M}_{\Lambda, 0} (k') = \{ \text{Dieudonn\'e lattices } A \subset \N' \text{ of height } 0 \ | \ \widetilde{x}(A \times A^*)  \subset A \times A^* \text{ for all } x \in \Lambda^\vee \}.$$
\end{definition}

Our main result for this section is:

\begin{proposition}\label{MS}
The bijection:
$$\mathcal{M}_0(k') \longleftrightarrow \mathcal{S}^+(k')$$
$$A \mapsto \frac{1}{p} \bigwedge_{W'}^2 A_1$$
restricts to a bijection:
$$\mathcal{M}_{\Lambda, 0} (k') \longleftrightarrow \mathcal{S}^+_{\Lambda}(k').$$
\end{proposition}

The proof of this will require a few lemmas. The first of them is:

\begin{lemma}\label{wisconsin1}
Let $A \subset \mathbb{N}'$ be a $W'$-lattice. For any $x \in \bigwedge_{W'_{\Q}}^2 \N'$, the corresponding homomorphism $x: \N'^* \rightarrow \N'$ takes $A^*$ to $A$ if and only if $x \in \bigwedge_{W'}^2 A$: 
$$\bigwedge_{W'}^2 A = \{x \in \bigwedge_{W'_{\Q}}^2 \N' \  | \  x (A^*) \subset A \}.$$
\end{lemma}

\begin{proof}
For any $\sum_i a_i \wedge b_i \in \bigwedge_{W'}^2 A$ and $f \in A^*$, note that $\sum_i a_i \wedge b_i (f) = \sum_i f(a_i)b_i - f(b_i)a_i$ is in $A$, as all $f(a_i)$ and $f(b_i) \in W'$.  And so,
$$\bigwedge_{W'}^2 A \subset  \{x \in \bigwedge_{W'_{\Q}}^2 \N' \  | \  x(A^*)  \subset A \}.$$

For the other inclusion, take a basis $\{a_i\}_1^4$ of $A$ as a $W'$-module. Then $\{a_i\}_1^4$ is a basis of $\N'$ over $W'_{\Q}$, so let $\{g_i\}_1^4$ be the dual basis of $\mathbb{N}^{'*}$, which is also a basis of $A^*$ as a free $W'$-module. Let $x$ be an element of $\bigwedge_{W'_{\Q}}^2 \N'$ that takes $A^*$ to $A$. Then $x = \sum_{i<j} w_{i,j} a_i \wedge a_j$ for some $w_{i,j} \in W'_{\Q}$. To prove the remaining inclusion, we need to show the $w_{i,j}$ are all in $W'$. 

Since $x$ takes $A^*$ to $A$, consider $x(g_1)$:
$$x(g_1) = \sum_{i<j} w_{i,j}a_i \wedge a_j (g_1) = \sum_{i<j} w_{i,j}g_1(a_i )a_j - w_{i,j}g_1(a_j)a_i = w_{1,2}a_2 + w_{1,3}a_3 + w_{1,4} a_4 $$

So $w_{1,2},  \ w_{1,3}$, and  $w_{1,4} \in W'$. By considering $x(g_k)$ for $2 \leq k \leq 4$, we have all $w_{i,j}$ are contained in $W'$.

\end{proof}

\begin{lemma}\label{wisconsin2}
Let $A \subset \mathbb{N}'$ be a $W'$ lattice. For any $t \in \bigwedge_{W'_{\Q}}^2 \N'^*$, the corresponding homomorphism $t: \N' \rightarrow \N'^*$ takes $A$ to $A^*$ if and only if $t \in \bigwedge_{W'}^2 A^*$: 
$$\bigwedge_{W'}^2 A^* = \{t \in \bigwedge_{W'_{\Q}}^2 \N'^{*} \  | \  t(A)  \subset A^* \}.$$
\end{lemma}

The proof of this is analogous to the proof of the previous lemma.

\begin{lemma}\label{wisconsin3}
Let $A \subset \N'$ be a $W'$-lattice of height 0. If $x \in \bigwedge_{W'}^2 A$, then $x^\star \in \bigwedge_{W'}^2 A^*$.
\end{lemma}

\begin{proof}
Let $A \subset \N'$ be a lattice of height 0. We have a similarly defined notion of height in $\N'^*$, using the element $\omega_1$, and $A^*$ is also of height 0. By Lemma ~\ref{isdual}, because $A \subset \N'$ is a lattice of height 0, the lattice $\bigwedge_{W'}^2 A \subset \bigwedge_{W'_{\Q}}^2 \N'$ is self-dual with respect to $[\cdot, \cdot]$. By exactly the same reasoning, because $A^* \subset \N'^*$ is of height 0, the lattice $\bigwedge_{W'}^2 A^* \subset \bigwedge_{W'_{\Q}}^2 \N'$ is self-dual with respect to $[\cdot, \cdot]_1$.

Let $x \in \bigwedge_{W'}^2 A$. Then, for all $y \in \bigwedge_{W'}^2 A^*$, the pairing $\{y,x\} \in W'$. So, for all $y \in \bigwedge_{W'}^2 A^*$,
$$\{x, y \} = [x^\star, y]^1 \in W'$$
and so $x^\star \in (\bigwedge_{W'}^2 A^*)^\vee = \bigwedge_{W'}^2 A^*$.
\end{proof}

The previous three lemmas combine to give the following result:

\begin{proposition}\label{verystrange}
Let $x \in \bigwedge_{W'_{\Q}}^2 \N'$ and let $A \subset \N'$ be a lattice of height 0. For any $x \in \bigwedge_{W'_{\Q}}^2 \N'$, the corresponding endomorphism $\widetilde{x}$ of $\underline{\N}' = \N' \times \N'^*$ takes $A \times A^*$ to itself if and only if $x \in \bigwedge_{W'}^2 A$:
$$\bigwedge_{W'}^2 A = \{ x \in \bigwedge_{W'_{\Q}}^2 \N' \ | \  \widetilde{x}( A \times A^*) \subset A \times A^* \}.$$
\end{proposition}

Now we are ready to prove Proposition ~\ref{MS}.

\begin{proof} (of Proposition ~\ref{MS}.)
We must show, for every Dieudonn\'e lattice $A \subset \N'$ of height 0, that $\widetilde{x}(A \times A^*)  \subset A \times A^*$ for all $x \in \Lambda^\vee$  if and only if  $\Lambda^\vee \subset \overline{\Phi}( \frac{1}{p} \bigwedge_{W'}^2 A_1 )$.

But, by Proposition  \ref{verystrange} , this is the same as showing: for every Dieudonn\'e lattice $A \subset \N'$ of height 0, that $\Lambda^\vee \subset \bigwedge_{W'}^2 A$   if and only if $\Lambda^\vee \subset \overline{\Phi}( \frac{1}{p} \bigwedge_{W'}^2 A_1 )$.

And, as $\overline{\Phi}( \frac{1}{p} \bigwedge_{W'}^2 A_1 ) = \bigwedge_{W'}^2 A$, the above statement is clear.

\end{proof}

We will later need a few more results similar to Proposition ~\ref{verystrange}. First, note that just as any $W'$-lattice $A \subset \N$ has an associated lattice $A_1 = F^{-1}(pA)$, any $W'$-lattice $B \subset \N'^*$ has an associated lattice $B_1 = F^{-1}(pB)$ (for the Frobenius operator $F$ on $\N'^*$). It is very important to notice that taking the dual lattice does \emph{not} commute with taking these associated lattices. In fact, if $A \subset \N'$ is a $W'$-lattice:
$$(A^*)_1 = (F^{-1}(A) )^* \neq (A_1)^*.$$

Similarly, for any $W'$-lattice $C \subset \underline{\N}'$, there is an associated lattice $C_1 = F^{-1}(pC)$ (for the Frobenius acting diagonally on $\underline{\N}' = \N' \times \N'^*$). If $A \subset \N'$ is a $W'$-lattice:
$$(A \times A^*)_1 = A_1 \times (A^*)_1 = F^{-1}(pA) \times (F^{-1}(A))^*.$$

So, given a $W'$-lattice $C \subset \underline{\N}'$ of the form $C = A \times A^*$, we have answered in Proposition \ref{verystrange} the question: what elements of $\bigwedge_{W'_{\Q}}^2 \N'$ takes $C$ to itself inside of $\underline{\N'}$? We can also ask: what elements of  $\bigwedge_{W'_{\Q}}^2 \N'$ take $C_1$ to itself? What elements of $\bigwedge_{W'_{\Q}}^2 \N'$ take $C_1$ to $C$? These questions will be answered in Propositions ~\ref{filtsum} and ~\ref{otherfiltsum}, respectively.

\begin{proposition}\label{filtsum}
Let $A \subset \N'$ be a $W'$-lattice of height 0. For any $x \in \bigwedge_{W'_{\Q}}^2 \N'$, the corresponding endomorphism $\widetilde{x}$ of $\underline{\N}' = \N' \times \N'^*$ takes $(A \times A^*)_1$ to itself if and only if $x \in   \frac{1}{p} \bigwedge_{W'}^2 A_1$:
$$ \frac{1}{p} \bigwedge_{W'}^2 A_1 = \{x \in \bigwedge_{W'_{\Q}}^2 \N' \ | \ \widetilde{x} ( (A \times A^*)_1 ) \subset   (A \times A^*)_1 \} .$$
\end{proposition}

\begin{proof}
The proof of this proposition relies on three observations. Let $A \subset \N'$ be a $W'$-lattices of height 0. Then:
\begin{enumerate}
\item{ $\frac{1}{p} \bigwedge_{W'}^2 A_1 = \{x \in \bigwedge_{W'_{\Q}}^2 \N' \  | \  x  ( (A^*)_1 ) \subset A_1 \} $ }
\item{ $ p \bigwedge_{W'}^2 ( A_1 )^* = \{t \in \bigwedge_{W'_{\Q}}^2 \N'^* \ | \ t ( A_1)  \subset (A^*)_1 \}$ }
\item{ Given $x \in \frac{1}{p} \bigwedge_{W'}^2 A_1$, then $x^\star  \in p \bigwedge_{W'}^2 ( A_1 )^*$.}
\end{enumerate}
These observations are analogous to the statements in Lemmas ~\ref{wisconsin1}, ~\ref{wisconsin2}, and ~\ref{wisconsin3}, respectively, and may be proven by similar methods. 

Recall that $A \subset \N'$ is a $W'$-lattice of height 0, and assume that $x \in \bigwedge_{W'_{\Q}}^2 \N'$ such that $\widetilde{x}( (A \times A^*)_1 ) \subset (A \times A^*)_1$. Then, in particular, $x( (A^*)_1 ) \subset A_1$. By the first observation, $x \in \frac{1}{p} \bigwedge_{W'}^2 A_1$.

On the other hand, assume that $x \in \frac{1}{p} \bigwedge_{W'}^2 A_1$. By the third observation, $x^\star \in  p \bigwedge_{W'}^2 ( A_1 )^*$. Then, by the first and second observations, $x( (A^*)_1 ) \subset A_1$ and $x^\star ( A_1) \subset (A^*)_1$. So, 
$$ \widetilde{x} ( (A \times A^*)_1 ) \subset   (A \times A^*)_1 .$$
\end{proof}

\begin{proposition}\label{otherfiltsum}
Let $A \subset \N'$ be a Dieudonn\'e lattice of height 0.   For any $x \in \bigwedge_{W'_{\Q}}^2 \N'$, the corresponding endomorphism $\widetilde{x}$ of $\underline{\N}' = \N' \times \N'^*$ takes $(A \times A^*)_1$ to $ A \times A^*$ if and only if $x \in \frac{1}{p} \bigwedge_{W'}^2 A_1 + \bigwedge_{W'}^2 A$:
$$ \frac{1}{p} \bigwedge_{W'}^2 A_1 + \bigwedge_{W'}^2 A = \{x \in \bigwedge_{W'_{\Q}}^2 \N' \ | \ \widetilde{x}( (A \times A^*)_1)  \subset A \times A^* \} .$$
\end{proposition}
\begin{proof}
Note that, because $A$ is a Dieudonn\'e lattice, $A_1 \subset A$ and $(A^*)_1 \subset A^*$.  Then, by Proposition \ref{filtsum} and Proposition~\ref{verystrange},
$$ \frac{1}{p} \bigwedge_{W'}^2 A_1 + \bigwedge_{W'}^2 A \subset \{x \in \bigwedge_{W'_{\Q}}^2 \N' \ | \  \widetilde{x}( (A \times A^*)_1)  \subset A \times A^* \}.$$

For the other inclusion, let $x \in \bigwedge_{W'_{\Q}}^2 \N'$ such that $ \widetilde{x}( (A \times A^*)_1)  \subset A \times A^*$. Then, in particular, $ x ( (A^*)_1 ) \subset A$. One can verify directly, using the elementary divisor theorem and the fact that $A$ is a ~\emph{ Dieudonn\'e} lattice, that:
$$ \frac{1}{p} \bigwedge_{W'}^2 A_ 1 + \bigwedge_{W'}^2 A = \{x \in \bigwedge_{W'_{\Q}}^2 \N' \  | \  x ( (A^*)_1 ) \subset A \}.$$

Therefore, $x \in  \frac{1}{p} \bigwedge_{W'}^2 A_1 + \bigwedge_{W'}^2 A$, and:
$$\{x \in \bigwedge_{W'_{\Q}}^2 \N' \ | \ \widetilde{x}( (A \times A^*)_1 ) \subset A \times A^* \} = \frac{1}{p} \bigwedge_{W'}^2 A_1 + \bigwedge_{W'}^2 A.$$

\end{proof}

In summary, given $A \subset \N'$ a Dieudonn\'e lattice of height 0, we have determined which $x \in \bigwedge_{W'_{\Q}}^2 \N'$ correspond to special endomorphisms $\widetilde{x}$ of $\underline{\N}'$ that take $A \times A^*$ to itself, take $(A \times A^*)_1$ to itself, or take $(A \times A^*)_1$ to $A \times A^*$. These observations will be very important to the arguments of  Section ~\ref{algebraicallydefined}. There is one remaining observation that will also be useful:

\begin{corollary}\label{filtrationsagain}
For any vertex lattice $\Lambda \subset V^{\Phi}$, we have defined:
$$\mathcal{M}_{\Lambda, 0} (k') = \{  A \in \mathcal{M}_0(k')  \ | \ \widetilde{x}(A \times A^*)  \subset A \times A^* \text{ for all } x \in \Lambda^\vee \}.$$
The set $\mathcal{M}_{\Lambda, 0} (k)$ can also be described as:
$$\mathcal{M}_{\Lambda, 0} (k') = \{  A \in \mathcal{M}_0(k') \ | \ \widetilde{x}( (A \times A^*)_1 )  \subset (A \times A^*)_1 \text{ for all } x \in \Lambda^\vee \}.$$
\end{corollary}

\begin{proof}

The key observation is, if $\Lambda$ is a vertex lattice and $L$ is a \emph{special} lattice, then:
$$\Lambda^\vee \subset \overline{\Phi}(L) \text{ if and only if } \Lambda^\vee \subset L.$$

Fix $\Lambda$ a vertex lattice and  let  $A$ be a Dieudonn\'e lattice of height 0. We have shown in Proposition ~\ref{MS} that $\widetilde{x}(A \times A^*)  \subset A \times A^*$   for all  $x \in \Lambda^\vee$ if and only if:
$$\Lambda^\vee \subset \overline{\Phi}( \frac{1}{p} \bigwedge_{W'}^2 A_1).$$

As $ \frac{1}{p} \bigwedge_{W'}^2 A_1$ is a special lattice, this is equivalent to the condition that:
$$\Lambda^\vee \subset  \frac{1}{p} \bigwedge_{W'}^2 A_1.$$

Finally, by Proposition ~\ref{filtsum}, the above condition holds if and only if  $\widetilde{x}( (A \times A^*)_1 )  \subset (A \times A^*)_1$   for all  $x \in \Lambda^\vee$. Therefore, the two sets above are equivalent descriptions of $\mathcal{M}_{\Lambda, 0} (k')$.

\end{proof}

\subsection{A Locus on Which Endomorphisms are Integral}

In this section we will define a closed subscheme $\mathcal{N}_{\Lambda, 0}$ of $\mathcal{N}_{0,red}$ with the property that $\mathcal{N}_{\Lambda,0}(k')$ is in bijection with $\mathcal{M}_{\Lambda, 0}(k') $ and therefore with $\mathcal{S}^+_{\Lambda}(k')$. To do this, we must first record a few observations concerning the slope-zero operator $\Phi$, the Hodge star operator, and the endomorphisms $\widetilde{x}$.

Recall that we have a slope-zero operator $\Phi$ on $V = \bigwedge_{W_{\Q}}^2 \N$ (and also on $\bigwedge_{W'_{\Q}}^2 \N'$) defined by:
$$\Phi(a \wedge b) = \frac{1}{p} F(a) \wedge F(b).$$

Using the Frobeius  operator on $\N^*$, we also have a slope zero operator on $\bigwedge_{W_{\Q}}^2 \N^*$ (and also on $\bigwedge_{W'_{\Q}}^2 \N'^*$) defined by:
$$\Phi(f \wedge g) = \frac{1}{p} F(f) \wedge F(g).$$

\begin{lemma}\label{phiandF}
With respect to the  injections $\bigwedge_{W_{\Q}}^2 \N \rightarrow \Hom(\N^*, \N)$ and $\bigwedge_{W_{\Q}}^2 \N^* \rightarrow \Hom(\N, \N^*)$, we have the following identities:
$$\Phi(a \wedge b) \circ F = F \circ a \wedge b $$
$$\Phi(f \wedge g) \circ F = F \circ f \wedge g $$
for all $a \wedge b \in \bigwedge_{W_{\Q}}^2 \N$ and for all $f \wedge g \in \bigwedge_{W_{\Q}}^2 \N^*$. Note that the first equality takes place in $\Hom(\N^*, \N)$ and the second equality takes place in $\Hom(\N, \N^*)$.
\end{lemma}

\begin{proof}
Let $a \wedge b \in \bigwedge_{W_{\Q}}^2 \N$ and let $f \in \N^*$. Then,
\begin{align*}
F \circ a \wedge b (f) &= f(a)^\sigma F(b) - f(b)^\sigma F(a) \\
&= \frac{1}{p} ( f( V \circ F (a) )^\sigma F(b) - f( V \circ F (b) )^\sigma F(a) ) \\
&= \frac{1}{p}( (Ff)(F(a) ) F(b) - (F f)(F(b) ) F(a) ) \\
&= \Phi(a \wedge b) \circ F (f) 
\end{align*}



The proof of the second identity is very similar.

\end{proof}

\begin{lemma}\label{phiandstar}
For any $x \in \bigwedge_{W_{\Q}}^2 \N$,
$$(\Phi(x) )^\star = \Phi( x^\star )$$
\end{lemma}

\begin{proof}

Recall that, because $\Phi$ is slope-zero and $\sigma$-semilinear:
$$[\Phi(x), \Phi(y)] = [x,y]^\sigma$$
for any $x, y \in \bigwedge_{W_{\Q}}^2 \N$. Similarly, for any $t,s \in \bigwedge_{W_{\Q}}^2 \N^*$,
$$[\Phi(t), \Phi(s)]_1 = [t,s]^\sigma_1.$$

We will also need the fact that $\{\Phi(x), \Phi(t) \}= \{x,t\}^\sigma$. (This may be checked directly for $x \in \bigwedge_{W_{\Q}}^2 \N$ of the form $a \wedge b$ and $t \in \bigwedge_{W_{\Q}}^2 \N^*$ of the form $f \wedge g$.)

With these observations, we can prove the lemma. Let $x \in \bigwedge_{W_{\Q}}^2 \N$, and consider any $t \in \bigwedge_{W_{\Q}}^2 \N^*$. Then,
$$[ \Phi(x)^\star, \Phi(t) ] _1 = \{ \Phi(x), \Phi(t) \} = \{x, t\}^\sigma.$$
And,
$$[\Phi(x^\star), \Phi(t)]_1 = [x^\star, t]^\sigma_1 = \{x,t \}^\sigma.$$
Then, as $\Phi$ is a bijection on $\bigwedge_{W_{\Q}}^2 \N^*$, and $[\cdot,\cdot]_1$ is nondegenerate, we have $(\Phi(x))^\star = \Phi(x^\star)$.

\end{proof}

\begin{proposition}\label{commutewithFrob}
Let $x \in \bigwedge_{W_{\Q}}^2 \N$. If $\Phi(x) = x$, then $F \circ \widetilde{x} = \widetilde{x} \circ F$ as elements of $\End(\underline{\N})$.\end{proposition}

\begin{proof}
Let $x \in  \bigwedge_{W_{\Q}}^2 \N$ such that $\Phi(x) = x$. Then, by Lemma ~\ref{phiandF}, $F \circ x = x \circ F$ as elements of $\Hom(\N^*, \N)$. 

By Lemma ~\ref{phiandstar}, 
$$\Phi(x^\star) = (\Phi(x))^\star = x^\star.$$

Then, again  by Lemma ~\ref{phiandF}, $F \circ x^\star = x^\star \circ F$ as elements of $\Hom(\N, \N^*)$. 

Finally, because $\widetilde{x}$ as an endomorphism of $\N$ is defined as $x^\star$ and $x$ on the two components $\N$ and $\N^*$, the Frobenius operator on $\underline{\N}$ is also defined on these two components, and $x^\star$ and  $x$ commute with Frobenius as elements of $\Hom(\N, \N^*)$ and $\Hom(\N^*, \N)$, we have:
$$F \circ \widetilde{x} = \widetilde{x} \circ F$$ 
as elements of $\End(\underline{\N})$.

\end{proof}

We are now ready to define the closed subschemes $\mathcal{N}_{\Lambda, 0} \subset \mathcal{N}_{0,red}$. Given $G$ a p-divisible group, we will use the notation $G^*$ for the dual $p$-divisible group. Recall that, if $S$ is a $W$-scheme on which $p$ is locally nilpotent, $\mathcal{N}_0(S)$ parametrizes pairs $(G, \rho)$, where $G$ is a supersingular $p$-divisible group over $S$ of height 4 and dimension 2, and 
$$\rho: G_{S_0} \rightarrow \mathbb{G}_{S_0}$$
is a quasi-isogeny of height 0. (Here, $S_0 = S \times_{\Spec(W) } \Spec(k)$.) Given such an $S$-point $(G, \rho)$, we can also define a quasi-isogeny:
$$\underline{\rho}: (\rho, (\rho^*)^{-1}): G_{S_0} \times_{S_0} G_{S_0}^* \rightarrow \mathbb{G}_{S_0} \times_{S_0} \mathbb{G}_{S_0}^*.$$

Any element $x \in V^\Phi$ defines a $W_{\Q}$-linear map $\widetilde{x}: \underline{\N} \rightarrow \underline{\N}$ that commutes with Frobenius, so defines an element $\widetilde{x} \in \End^0(\underline{\G})$. So, for any $(G, \rho) \in \mathcal{N}_0(S)$, we have a quasi-endomorphism:
$$ \underline{\rho}^{-1} \circ \widetilde{x} \circ \underline{\rho} \in \End^0( G_{S_0} \times_{S_0} G_{S_0}^* ).$$

We can ask: for what $(G, \rho) \in \mathcal{N}_0(S)$ will $ \underline{\rho}^{-1} \circ \widetilde{x} \circ \underline{\rho}$ be  an endomorphism of $  G_{S_0} \times_{S_0} G_{S_0}^* $, as opposed to just a quasi-endomorphism? This is the definition of $\mathcal{N}_{\Lambda, 0}$:

\begin{definition}
Let $\Lambda \subset V^{\Phi}$ be a vertex lattice. Let $S$ be a $W$-scheme on which $p$ is locally nilpotent. Then, we define:
$$\widetilde{\mathcal{N}_{\Lambda, 0}}(S) = \{ (G, \rho) \in \mathcal{N}_0(S) \ | \  \underline{\rho}^{-1} \circ \widetilde{x} \circ \underline{\rho} \in \End( G_{S_0} \times_{S_0} G_{S_0}^* )  \text{ for all } x \in \Lambda^\vee \}.$$
(As before,  $S_0 = S \times_{\Spec(W) } \Spec(k)$.)
This defines a closed formal subscheme $\widetilde{\mathcal{N}_{\Lambda, 0}} \subset \mathcal{N}$.   We define $\mathcal{N}_{\Lambda,0}$ to be the reduced $k$-scheme underlying $\widetilde{\mathcal{N}_{\Lambda, 0}}$.
\end{definition}

\begin{proposition}\label{NM}
Let $\Lambda \subset V^\Phi$ be a vertex lattice. The bijection:
$$\mathcal{N}_{0,red}(k') \longleftrightarrow \mathcal{M}_0(k')$$
$$(G, \rho) \mapsto \rho(M) \subset \N'$$
restricts to a bijection:
$$\mathcal{N}_{\Lambda, 0}(k') \longleftrightarrow \mathcal{M}_{\Lambda,0}(k').$$
\end{proposition}

\begin{proof}

Recall that:

$$\mathcal{M}_{\Lambda, 0} (k') = \{ A \in \mathcal{M}_0(k') \ | \ \widetilde{x}(A \times A^*)  \subset A \times A^* \text{ for all } x \in \Lambda^\vee \}$$
and
$$\mathcal{N}_{\Lambda, 0}(k') = \{ (G, \rho) \in \mathcal{N}_0(k') \ | \  \underline{\rho}^{-1} \circ \widetilde{x} \circ \underline{\rho} \in \End( G \times_{k'} G^* )  \text{ for all } x \in \Lambda^\vee \}.$$

Given any $(G, \rho) \in \mathcal{N}_0(k')$, let $(M, M_1, \Psi)$ be the associated $W'$-window over $k'$, and let $N = M_{\Q}$. In our situation, $M_1 = \Psi^{-1}(pM)$. The quasi-isogeny $\rho$ induces an isomorphism of $W'_{\Q}$ vector spaces $\rho: N \rightarrow \mathbb{N}$, and the Dieudonn\'e lattice in $\mathcal{M}_{\Lambda, 0} (k') $ associated to $(G, \rho)$ is $\rho(M)$.

Let $N^* = \Hom(N, W'_{\Q})$ and let $M^*$ be the dual $W'$ module to $M$. There is a dual operator, also denoted $\Psi$, on $N^*$. With this notation, the window associated to $G \times_{k'} G^*$ is 
$$(M \times M^*, (M \times M^*)_1, \Psi)$$
where $\Psi$ acts diagonally and  $(M \times M^*)_1 = \Psi^{-1}(p (M \times M^*) )$. Let $\Lambda \subset V^{\Phi}$ be a vertex lattice and let $x \in \Lambda^\vee$. Then we have $W'_{\Q}$-linear map:
$$ \underline{\rho}^{-1} \circ \widetilde{x} \circ \underline{\rho}: N \times N^* \rightarrow N \times N^*.$$

By Zink's theory of windows ~\cite{Windows}, this map will induce an endomorphism of the $p$-divisible group $G \times_{k'} G^*$ if and only if the following three conditions all hold:
\begin{enumerate}
\item{ $\underline{\rho}^{-1} \circ \widetilde{x} \circ \underline{\rho}$ commutes with $\Psi$}
\item{  $\underline{\rho}^{-1} \circ \widetilde{x} \circ \underline{\rho}(M \times M^*) \subset M \times M^*$ }
\item{  $\underline{\rho}^{-1} \circ \widetilde{x} \circ \underline{\rho}( (M \times M^*)_1 ) \subset (M \times M^*)_1$.}
\end{enumerate}

The first condition holds for any $x \in V^\Phi$, because by Proposition ~\ref{commutewithFrob} the fact that $x$ is $\Phi$-invariant causes $\widetilde{x}$ to commute with Frobenius, and $\underline{\rho}$ intertwines the actions of $F$ on $\underline{\N}'$ and $\Psi$ on $N \times N^*$. 

Note that $\underline{\rho}$ identifies $M \times M^*$ with $\rho(M) \times (\rho(M))^*$ and, because $\underline{\rho}$ intertwines the actions of $\Psi$ and $F$, $\underline{\rho}$ also identifies $(M \times M^*)_1$ with $(\rho(M) \times (\rho(M))^*)_1$. With these observations,  $x \in \Lambda^\vee$ will induce an endomorphism $ \underline{\rho}^{-1} \circ \widetilde{x} \circ \underline{\rho}$ of $G \times_{k'} G^*$ if and only if the following two conditions hold:
\begin{enumerate}
\item{  $\widetilde{x} (\rho(M) \times \rho(M)^*) \subset \rho(M) \times \rho(M)^*$ }
\item{  $\widetilde{x} ( (\rho(M) \times \rho(M)^*  )_1 ) \subset (\rho(M) \times \rho(M)^*)_1$.}
\end{enumerate}
   
And, by Corollary ~\ref{filtrationsagain}, these two conditions hold for all $x \in \Lambda^\vee$ if and only if $\rho(M) \in \mathcal{M}_{\Lambda, 0} (k')$.

\end{proof}

\begin{corollary}\label{NMS}
Let $\Lambda \subset V^{\Phi}$ be a vertex lattice. The bijections:
$$\mathcal{N}_{0,red}(k') \quad \quad \longleftrightarrow \quad \quad  \mathcal{M}_{0}(k') \quad \quad  \longleftrightarrow  \quad \quad \mathcal{S}^+(k')$$
$$(G, \rho) \mapsto \rho(M) \subset \N', \quad A \mapsto \frac{1}{p} \bigwedge_{W'}^2 A_1 \subset V'$$
Identify:
$$\mathcal{N}_{\Lambda, 0}(k') \longleftrightarrow \mathcal{M}_{\Lambda,0}(k') \longleftrightarrow \mathcal{S}_{\Lambda}^+(k').$$
\end{corollary}

\begin{proof}
This is the combination of Proposition ~\ref{NM} and Proposition ~\ref{MS}.
\end{proof}

Note that, by combing Proposition \ref{vollaardprop} with Corollary ~\ref{NMS}, we have:
$$\mathcal{N}_{0,red} = \bigcup_{\Lambda} \mathcal{N}_{\Lambda,0}$$
where the index is over all vertex lattices $\Lambda$. This is \emph{not} a disjoint union:  $\mathcal{N}_{\Lambda_1} \subset \mathcal{N}_{\Lambda_2}$ whenever $\Lambda_1 \subset \Lambda_2$. 


Our next objective is to understand the schemes $\mathcal{N}_{\Lambda,0}$. The conclusion (in Section ~\ref{Computingimage})  will be very nice: depending on if $\Lambda$ is a vertex lattice of type 2 or type 4, $\mathcal{N}_{\Lambda,0}$ will either be a single point or will be isomorphic to $\mathbb{P}^1$. As a first step towards this, we have the following:

\begin{proposition}\label{projective}
For any vertex lattice $\Lambda$, the $k$-scheme $\mathcal{N}_{\Lambda, 0}$ is projective
\end{proposition}

\begin{proof}
Let $\Lambda$ be a vertex lattice. Since $\Lambda \otimes_{\Z_p} W \subset \bigwedge_{W_{\Q}}^2 \N$ is a $W$-lattice, there is some sufficiently large positive integer $a$ such that:
$$\Lambda \otimes_{\Z_p} W \subset p^{-a} \bigwedge_{W'}^2 \M.$$
By Corollary ~\ref{NMS}, for any $(G, \rho) \in \mathcal{N}_{\Lambda,0}(k)$, 
$$\Lambda^\vee \subset \bigwedge_{W'}^2 \rho(M).$$

Noting that $\bigwedge_{W'}^2 \M$ and $\bigwedge_{W'}^2 \rho(M)$ are self-dual, we can combine these to produce the following chain condition:
$$p^a \bigwedge_{W'}^2 \M \subset \Lambda^\vee \otimes_{\Z_p} W \subset \bigwedge_{W'}^2 \rho(M) \subset \Lambda \otimes_{\Z_p} W \subset p^{-a} \bigwedge_{W'}^2 \M.$$
Using the fact that $\rho(M) \subset \N$ is a lattice of height 0, an elementary computation based on the chain condition above produces the following bounds:
$$p^{3a} \M \subset \rho(M) \subset p^{-3a} \M.$$
This holds for any $(G, \rho) \in \mathcal{N}_{\Lambda,0}(k)$, and the integer $a$ is independent of the choice of $(G, \rho)$. Then, by ~\cite{RZ} Corollary 2.29, $\mathcal{N}_{\Lambda,0}$ is a closed subscheme of a projective scheme. 
\end{proof}

\section{A Subvariety of an Orthogonal Grasssmanian}\label{Section5}

Our objective is to understand the geometric structure of $\mathcal{N}_{\Lambda,0}$. So far, we only have a bijection of sets between $\mathcal{N}_{\Lambda,0}(k')$ and $\mathcal{S}^+_{\Lambda}(k')$. However, $\mathcal{S}^+_{\Lambda}(k')$ can naturally be identified with the set of $k'$ points of a subvariety of an orthogonal Grassmanain. In this section, we'll describe that subvariety and show that the resulting bijection between the $k'$ points of $\mathcal{N}_{\Lambda,0}$ and the $k'$-points of this subvariety is more than just a bijecton on points: it results from an isomorphism between these two schemes.

\subsection{Special Lattices and Grassmanians}

For this section, fix a vertex lattice $\Lambda \subset V^{\Phi}$ of type $2d$. (Recall that $d = 1$ or $2$.) Define:
$$\Omega_0 = \Lambda^\vee / \Lambda.$$

This is an  $\F_p$-vector space of dimension $2d$. Define an $\F_p$-valued symmetric bilinear form $[\cdot,\cdot]_{\Lambda}$ on $\Omega_0$ by:
$$[ [x] ,   [ y]  ]_{\Lambda} = p [ x , y ] \mod p$$
for any $[x], [y] \in \Omega_0$. Note that as $\Lambda^\vee \subset \Lambda \subset \frac{1}{p} \Lambda^\vee$, for any $x,y \in \Lambda$, we have that $p[x,y] \in \Z_p$, so the above description of $[ \cdot, \cdot]_{\Lambda}$ is valid. The nondegeneracy of $[ \cdot, \cdot ]$ implies the nondegeneracy of $[ \cdot, \cdot]_{\Lambda}$ and the fact that $V^{\Phi}$ has Hasse invariant -1 implies that $(\Omega_0, [\cdot, \cdot]_{\Lambda})$ is nonsplit.

 Let $\Omega = \Omega_0 \otimes_{\F_p} k$, and extend scalars to get a $k$-valued symmetric bilinear form, also denoted $[ \cdot, \cdot]_{\Lambda}$, on $\Omega$. Let the symbol $\Phi$ denote Frobenius acting on $\Omega$.

Let $\mathrm{OGr}(\Omega)(r)$ be the scheme such that, if $S$ is any  $k$-scheme, $\mathrm{OGr}(\Omega)(r)(S)$ parametrizes all totally isotropic local $\mathcal{O}_S$-module direct summands
$$\mathcal{L} \subset \Omega \otimes_k \mathcal{O}_s$$
  of rank  $r$. So, in particular, $\mathrm{OGr}(\Omega)(d)(k)$ is the moduli space of Lagrangian subspaces of $\Omega$.

  Note that, if $k'$ over $k$ is not algebraically closed and $\mathcal{L}$ is a $k'$-point of $\mathrm{OGr}(r)$, then $\Phi( \mathcal{L})$ does \emph{not} necessary correspond to a $k'$-point of $\mathrm{OGr}(r)$, but the $k'$-vector space generated by $\Phi( \mathcal{L})$, denoted $\overline{\Phi}(\mathcal{L})$, does define a $k'$-point of $\mathrm{OGr}(r)$.  Because $\Omega$ is defined by $\Omega_0$ over $\F_p$, there is a natural action of Frobenius on $\mathrm{OGr}(r)$, and on $k'$-points this action sends $\mathcal{L}$ to $\overline{\Phi}(\mathcal{L})$. To emphasize this, we'll denote the action on $S$-points also as $\mathcal{L} \mapsto \overline{\Phi}(\mathcal{L})$. 

Similarly, let $\mathrm{OGr}(\Omega)(d-1, d)$ be the scheme such that,  if $S$ is any  $k$-scheme, $\mathrm{OGr}(d-1, d)(S)$ parametrizes all  flags of totally isotropic local $\mathcal{O}_S$-module direct summands 
$$\mathcal{L}_{d-1} \subset \mathcal{L}_d \subset \Omega \otimes_k \mathcal{O}_s $$
where $\mathcal{L}_{d-1}$ is of rank $d-1$ and $\mathcal{L}_d$ is of rank $d$.

Then $\mathrm{OGr}(\Omega)(d-1, d)$ has two connected components: two flags are in the same component if and only if there is an orthogonal transformation of determinant 1 carrying the first flag to the second flag. These two connected components are interchanged by any orthogonal transformation of determinant -1. As in \cite{GU22}, one can check in coordinates that the action of Frobenius on $\mathrm{OGr}(\Omega)(d-1, d)$ interchanges the two components.  Label these two connected components as $\mathrm{OGr}(\Omega)^+(d-1, d)$ and $\mathrm{OGr}(\Omega)^-(d-1, d)$.

Now, let $\mathcal{X}_{\Lambda}$ be the reduced closed subscheme of $\mathrm{OGr}(\Omega)(d)$ with $k'$-points:
$$\mathcal{X}_{\Lambda}(k') = \{ \mathcal{L} \in \mathrm{OGr}(\Omega)(d)(k') \ | \ \dim_{k'}( \mathcal{L} + \overline{\Phi}(\mathcal{L} ) ) =  d+1 \}.$$

Just as in ~\cite{GU22}, there is a closed immersion 
$$\mathcal{X}_{\Lambda} \rightarrow \mathrm{OGr}(d-1, d)$$
$$\mathcal{L} \mapsto \mathcal{L} \cap \overline{\Phi}(\mathcal{L} ) \subset \mathcal{L}.$$

And so $\mathcal{X}_{\Lambda}$ decomposes into two open and closed subschemes, which we'll denote $\mathcal{X}_{\Lambda}^\pm$. The action of Frobenius also interchanges these two components. Using the fact that $\Omega$ contains no $\Phi$-invariant Lagrangian subspaces, it will be useful to consider $\mathcal{X}^{\pm}_{\Lambda}$ as the scheme with $k'$-points:
 $$\mathcal{X}^{\pm}_{\Lambda}(k') = \{ \mathcal{L}_{d-1} \subset \mathcal{L}_d  \in \mathrm{OGr}(\Omega)^{\pm}(d-1, d)(k')  \ |  \ \mathcal{L}_{d-1} \subset \overline{\Phi}(\mathcal{L}_d) \}.$$

\begin{proposition}\label{Xispt}
Let $\Lambda$ be a vertex lattice of type 2. Then  $\mathcal{X}^{+}_{\Lambda}$ and $\mathcal{X}^-_{\Lambda}$ are each a single point.
\end{proposition}

\begin{proof}
Let $\Lambda$ be a vertex lattice of type 2, so $d=1$. Then $\Omega = \Lambda/\Lambda^\vee \otimes_{\F_p} k$ is a 2-dimensional quadratic space containing an isotropic vector, so is a hyperbolic plane. As $\Omega$ has exactly two Lagrangian subspaces, $\mathrm{OGr}(\Omega)(1)$ consists of two points. Since $V^{\Phi}$ has Hasse invariant -1, neither of these Lagrangians may be fixed by Frobenius, so $\Phi$ interchanges the two points of $\mathrm{OGr}(\Omega)(1)$. 

For either of these Lagrangian subspaces $\mathcal{L}$, because $\mathcal{L} \neq \overline{\Phi}(\mathcal{L})$ and the quadratic form on $\Omega$ is nondegenerate, the dimension of  $\mathcal{L} + \overline{\Phi}(\mathcal{L})$ is exactly 3. As  $\mathcal{X}_{\Lambda}$ is the subscheme of $\mathrm{OGr}(\Omega)(1)$ parametrizing Lagrangians $\mathcal{L}$ such that $\mathcal{L} + \overline{\Phi}(\mathcal{L})$ has dimension 3,  $\mathcal{X}_{\Lambda}$ consists of two points, interchanged by Frobenius, and so each $\mathcal{X}^{\pm}_{\Lambda}$ consists of a single point.

\end{proof}

\begin{proposition}\label{type4P1}
Let $\Lambda$ be a vertex lattice of type $4$. Then each $\mathcal{X}^{\pm}_{\Lambda}$ is isomorphic to $\mathbb{P}^1$, as a variety over $k$.
\end{proposition}

\begin{proof}
Let $\Lambda$ be a vertex lattice of type 4, so $d=2$. Then $\Omega_0 = \Lambda^\vee/\Lambda$ is a nondegenerate, nonsplit quadratic space over $\F_p$ of dimension 4.  If we let $Q_{\Lambda}$ denote the quadratic form on $\Omega_0$, then (up to isomorphism) there is some nonsquare $D \in \F_p$ such that:
$$Q_{\Lambda}(x,y,z,w) = xy + z^2 - Dw^2.$$

Then, $\mathrm{OGr}(\Omega)(1)$ (the moduli space of isotropic lines in $\Omega$) is isomorphic to the quadric defined by $Q_{\Lambda}$ inside $\mathbb{P}^3$, which is isomorphic over $k$ to $\mathbb{P}^1 \times \mathbb{P}^1$. Concretely, if $\Delta \in k$ is a square root of $D$, one possible isomorphism is:
$$\psi:  \mathbb{P}^1 \times \mathbb{P}^1 \xrightarrow{\sim} \mathrm{OGr}(\Omega)(1) $$
$$\psi([a:b], [c:d]) = [ac: bd: \frac{1}{2}(ad+bc) : \frac{1}{2\Delta} (ad-bc)]$$
when we view $\mathrm{OGr}(\Omega)(1) $ as a closed subvariety of $\mathbb{P}^3$. Note that, while both $\mathbb{P}^1 \times \mathbb{P}^1$ and $\mathrm{OGr}(\Omega)(1) $ have natural $\F_p$-structures, this morphism only descends to $\F_{p^2}$, not $\F_p$. We will make this morphism $\Phi$-equivariant by defining a twisted action of $\Phi$ on $\mathbb{P}^1 \times \mathbb{P}^1$ by:
$$\Phi([a:b], [c:d]) = (\Phi[c:d], \Phi[a:b]).$$

It is an elementary fact that every isotropic line in $\Omega$ is contained in exactly two isotropic planes. In terms of the isomorphism above, considering an isotropic line as $\psi([a:b], [c:d])$ for some $([a:b], [c:d]) \in \mathbb{P}^1 \times \mathbb{P}^1$, the two isotropic planes containing this line are $\psi([a:b], \mathbb{P}^1)$ and $\psi( \mathbb{P}^1, [c:d])$. Then (up to possibly relabelling the components) each of $\mathrm{OGr}(\Omega)^{\pm}(1,2)$ are isomorphic to $\mathbb{P}^1 \times \mathbb{P}^1$, by the maps:
$$\mathbb{P}^1 \times \mathbb{P}^1 \xrightarrow{\sim} \mathrm{OGr}^+(\Omega)(1,2) $$
$$([a:b], [c:d] ) \mapsto \psi([a:b], [c:d] ) \subset \psi([a:b], \mathbb{P}^1 )$$
$$\mathbb{P}^1 \times \mathbb{P}^1 \xrightarrow{\sim} \mathrm{OGr}^-(\Omega)(1,2) $$
$$([a:b], [c:d] ) \mapsto \psi([a:b], [c:d] ) \subset \psi( \mathbb{P}^1, [c:d] ).$$

Note that $\Phi( \psi([a:b], \mathbb{P}^1 ) ) = \psi( \mathbb{P}^1, \Phi[a:b])$, and so
$$\mathcal{X}^{+}_{\Lambda} = \{ \mathcal{L}_{1} \subset \mathcal{L}_2  \in \mathrm{OGr}(\Omega)^{+}(1, 2)  \ | \  \mathcal{L}_{1} \subset \overline{\Phi}(\mathcal{L}_2) \}$$
can be identified with $\{ ( [a:b], \Phi[a:b]) \} \subset  \mathbb{P}^1 \times \mathbb{P}^1 .$

Similarly, $\mathcal{X}^{-}_{\Lambda}$ can be identified with $\{ (  \Phi[c:d], [c:d] ) \} \subset \mathbb{P}^1 \times \mathbb{P}^1$. Therefore, both $\mathcal{X}^{+}_{\Lambda}$ and $\mathcal{X}^{-}_{\Lambda}$ are isomorphic (over $\F_{p^2}$, so also over $k$) to $\mathbb{P}^1$.

\end{proof}

\begin{proposition}\label{SX}
For any field $k'$ over $k$, there is a bijection:
$$\mathcal{S}_{\Lambda}(k') \longleftrightarrow \mathcal{X}_{\Lambda}(k').$$
Given by:
$$L \mapsto \mathcal{L} := [L] \in \Lambda/\Lambda^\vee \otimes_{\F_p} k' =  \Omega \otimes_{k} k '.$$
\end{proposition}

\begin{proof}
Recall that $\mathcal{S}_{\Lambda}(k')$ is the set of special lattices $L \subset V'$ containing $\Lambda^\vee$. Because special lattices are self-dual, for any  $L \in \mathcal{S}_{\Lambda}(k')$, we have the following chain condition:

$$\Lambda^\vee \otimes_{Z_p} W' \subset L \subset \Lambda \otimes_{Z_p} W'.$$

As $\Omega \otimes_k k'  \cong (\Lambda \otimes_{\Z_p} W')/(\Lambda^\vee \otimes_{\Z_p} W')$, there is a bijection between subspaces $\mathcal{L}$ of $\Omega \otimes_k k'$ and $W'$-lattices $L$ satisfying the chain condition above. 

Note that, as suggested by the notation, this bijection is $\Phi$-equivariant for the action of the slope-zero operator $\Phi$ originally defined on $V'  = \bigwedge_{W'_{\Q}}^2 \N'$ and Frobenius $\Phi$ acting on $\Omega \otimes_{k} k'$. 

The two conditions for a $W'$-lattice $L \subset V$ to be a special lattice are:

\begin{enumerate}
\item{$L$ is self-dual with respect to $[\cdot, \cdot]$ }
\item{ $(L + \overline{\Phi}(L))/L$ has length 1. }
\end{enumerate}

The first condition is equivalent to $\mathcal{L}$ being a Lagrangian subspace of $\Omega$ (and so of dimension $d$) and the second condition is equivalent to $\dim_{k'}( \mathcal{L} + \overline{\Phi}(\mathcal{L}) ) = d+1$. 
\end{proof}

\begin{corollary}\label{mapsum}
For any field $k'$ over $k$ and vertex lattice $\Lambda$ of type $2d$, we have constructed a map on $k'$-points:
$$\mathcal{N}_{\Lambda, 0}(k')   \longrightarrow \mathcal{X}_{\Lambda}(k').$$
\end{corollary}

\begin{proof}
Combine Corollary ~\ref{NMS} and Proposition ~\ref{SX}. Concretely, this map can be described as follows if $(G, \rho) \in \mathcal{N}_{\Lambda,0}$, let $L = \frac{1}{p} \bigwedge_{W'}^2 (\rho(M) )_1 \subset V'$. This is a (very) special lattice containing $\Lambda^\vee$, so let $\mathcal{L} \subset \Omega \otimes_k k'$ be the resulting subspace. The map is then:
$$(G, \rho) \mapsto \mathcal{L}.$$
Note that this map factors through the set $\mathcal{S}_{\Lambda}^+(k')$.
\end{proof}

Recall that we have the subset of very special lattices $\mathcal{S}^+_{\Lambda}(k') \subset \mathcal{S}_{\Lambda}(k')$, and we also have two components $\mathcal{X}^{\pm}_{\Lambda}(k') \subset \mathcal{X}_{\Lambda}(k')$. As suggested by the notation, the map described above in Corollary ~\ref{mapsum} will not be a bijection, but we will show that it is the result of an isomorphism between $\mathcal{N}_{\Lambda,0}$ and one of the components $\mathcal{X}^{\pm}_{\Lambda}$.

\subsection{Algebraically Defined}\label{algebraicallydefined}

The goal of this section is to show that the map on $k'$-points described above in Corollary ~\ref{mapsum} is the result of a morphism of varieties $\mathcal{N}_{\Lambda,0} \rightarrow \mathcal{X}_{\Lambda}$.

The statement and proof of the following proposition follows the statement and  proof of Theorem 3.9 of ~\cite{GU22} very closely, and are included here for completion. The key difference is that in ~\cite{GU22}, the authors construct a map from (the dual of) a vertex lattice to a set of endomorphisms of their basepoint p-divisible group, while in this case (the dual of) a vertex lattice determines a set of endomorphisms of $\G \times \G^*$. The background necessary to discuss endomorphisms of $\G \times \G^*$ has been built up throughout this paper, including the properties of the Hodge star operator and the conclusions of Section \ref{Lattices Stabilized by Ends}.

\begin{proposition}\label{morphism}
Let $\Lambda$ be a vertex lattice. There is a morphism of $k$-schemes 
$$\mathcal{N}_{\Lambda,0} \rightarrow \mathcal{X}_{\Lambda}$$ 
inducing the map described in Corollary ~\ref{mapsum} on $k'$-points, for any field $k'$ over $k$.
\end{proposition}

\begin{proof}
Let $R$ be either a reduced $k$-algebra of finite type or a field over $k$, and let $(G, \rho) \in \mathcal{N}_{\Lambda,0}(R)$. By definition of $\mathcal{N}_{\Lambda,0}$, there is a map of $\Z_p$-modules:
$$\Lambda^\vee \rightarrow \mathrm{End}(G \times_R G^*)$$
$$x \mapsto  \underline{\rho}^{-1} \circ \widetilde{x} \circ \underline{ \rho}.$$

Let $\mathcal{D}$ be the covariant Grothendieck-Messing crystal of $G \times_R G^*$ evaluated at the trivial divided-power thickening $\mathrm{Spec}(R) \rightarrow \mathrm{Spec}(R)$. Then $\mathcal{D}$ is a locally free $R$-module, with an exact sequence:
$$0 \rightarrow \mathcal{D}_1 \rightarrow \mathcal{D} \rightarrow \mathrm{Lie}(G \times_R G^*) \rightarrow 0.$$

By the functoriality of the pair $\mathcal{D}_1 \subset \mathcal{D}$, we have induced morphisms of $R$-modules:
$$\psi: (\Lambda^\vee / p \Lambda^\vee ) \otimes_{\F_p} R \rightarrow \mathrm{End}_R( \mathcal{D} ) \quad \text{and} \quad \psi_1: (\Lambda^\vee / p \Lambda^\vee ) \otimes_{\F_p} R \rightarrow \mathrm{End}_R( \mathcal{D}_1 ).$$
Note that $\mathrm{ker}(\psi) \subset \mathrm{ker}(\psi_1)$. 

Recall that the $W_{\Q}$-valued pairing $[\cdot, \cdot]$ on $V$ induces a $\Q_p$-valued pairing also denoted $[\cdot, \cdot]$ on $V^{\Phi}$. As $\Lambda^\vee \subset V^{\Phi}$, we may restrict this pairing to a pairing on $\Lambda^\vee$. Further, because $\Lambda^\vee \subset \Lambda$, the pairing is $\Z_p$-valued on $\Lambda^\vee$. So, we may reduce modulo $p$ and extend scalars to $R$ to get an $R$-valued pairing on $ (\Lambda^\vee / p \Lambda^\vee ) \otimes_{\F_p} R .$ This pairing will also be denoted $[\cdot, \cdot]$. 

Also recall that, by Proposition ~\ref{Pairingiscomp}, for any $x,y \in \bigwedge_{W_{\Q}}^2 \N$,
$$[x,y] = \widetilde{x} \circ \widetilde{y} + \widetilde{y} \circ \widetilde{x}$$
where the right hand side is considered as multiplication by an element of $W_{\Q}$ as an endomorphism of $\underline{\N}$. After extending scalars to $R$, the same equality holds for $x,y \in (\Lambda^\vee/  p\Lambda^\vee) \otimes_{\F_p} R$, with the right hand side interpreted as the induced endomorphism of~$\mathcal{D}$.

In particular, if $x \in \mathrm{ker}(\psi_1)$ and $y \in (\Lambda^\vee/  p\Lambda^\vee) \otimes_{\F_p} R$, then $\widetilde{x} \circ \widetilde{y} + \widetilde{y} \circ \widetilde{x}$ is multiplication by an element of $R$, that is zero when restricted to $\mathcal{D}_1 \subset \mathcal{D}$. So,
$$[x,y] = \widetilde{x} \circ \widetilde{y} + \widetilde{y} \circ \widetilde{x} = 0$$
and therefore $\mathrm{ker}(\psi_1) \subset \mathrm{rad} ( (\Lambda^\vee/  p\Lambda^\vee) \otimes_{\F_p} R)$.

Multiplication by $p^{-1}$ induces an isomorphism:
$$\mathrm{rad} ( (\Lambda^\vee/  p\Lambda^\vee) \otimes_{\F_p} R) = (p \Lambda/ p \Lambda^\vee) \otimes_{\F_p} R \cong (\Lambda/ \Lambda^\vee) \otimes_{\F_p} R \cong \Omega \otimes_{k} R.$$
Let $\mathcal{L}^\# \subset \mathcal{K}$ be the images of $\mathrm{ker}(\psi) \subset \mathrm{ker}(\psi_1)$ under this isomorphism.

We claim that there is another distinguished $\mathcal{L} \subset \mathcal{K}$. To see this, first consider the case that $R=k'$ is a field extension of $k$. Then the point $(G, \rho) \in \mathcal{N}_{\Lambda}(k')$ corresponds to a Dieudonn\'e lattice $\rho(M) \subset \N'$, and the pair $\mathcal{D}_1 \subset \mathcal{D}$ can be canonically identified with:
$$(\underline{\rho}( M \times M^* ) )_1/ p(\underline{\rho} (M \times M^* )  ) \subset \underline{\rho}( M \times M^* ) / p(\underline{\rho} (M \times M^* )  ). $$

Under these identifications,
$$\mathcal{L}^\# = \{ x \in ( \Lambda / \Lambda^\vee) \otimes_{\F_p} k' \  | \  \widetilde{x}(  \underline{\rho}( M \times M^* ) ) \subset  \underline{\rho}( M \times M^* ) \}$$
$$\mathcal{K} = \{ x \in ( \Lambda / \Lambda^\vee) \otimes_{\F_p} k' \  | \  \widetilde{x} ( (\underline{\rho}( M \times M^* ) )_1 ) \subset  \underline{\rho}( M \times M^* ) \}.$$

Let $L^\#$ be the lattice in $\bigwedge_{W'_{\Q}}^2 \N'$ corresponding to the subspace $\mathcal{L}^\#$ and let $K$ be the lattice corresponding to $\mathcal{K}$. Then,
$$L^\# = \{ x \in \bigwedge_{W'_{\Q}}^2 \N' \   | \  \widetilde{x} ( \underline{\rho}( M \times M^* ) ) \subset \underline{\rho}( M \times M^* ) \}$$
$$K = \{ x \in \bigwedge_{W'_{\Q}}^2 \N'  \ | \ \widetilde{x} ( (\underline{\rho}( M \times M^* ) )_1 ) \subset \underline{\rho}( M \times M^* ) \}.$$

Noting that $\underline{\rho}(M \times M^*) = \rho(M) \times \rho(M)^*$, by Propositions ~\ref{verystrange} and ~\ref{otherfiltsum} respectively, we have:
$$L^\# =  \bigwedge_{W'}^2 \rho(M) \quad \quad \text{and} \quad \quad K = \frac{1}{p} \bigwedge_{W'}^2 (\rho(M))_1 + \bigwedge_{W'}^2 \rho(M).$$

Let $L = \bigwedge_{W'}^2 \rho(M)_1$, and let $\mathcal{L}$ be the subspace corresponding to $L$. Our objective is to define $\mathcal{L}$ functorially from $\mathcal{L}^+$ and $\mathcal{K}$ without reference to Dieudonn\'e lattices (which we may only use over a field).

By Proposition ~\ref{SX}, $\mathcal{L}^\# \subset \Omega \otimes_k k'$ is totally isotropic of dimension $d$, and $\mathcal{K}$ has dimension $d+1$. Because $\mathcal{K}^\perp$ is $d-1$ dimensional, and $\mathcal{K}^\perp \subset \mathcal{L}^\# \subset \mathcal{K}$, we can view $\mathcal{K}/\mathcal{K}^\perp$ as a two-dimensional quadratic space. Since $\mathcal{L}^\#$ corresponds to one isotropic line in $\mathcal{K}/\mathcal{K}^\perp$, there is a unique subspace of $\Omega \otimes_k k'$ corresponding to the other isotropic line. This other subspace is $\mathcal{L}$.


Just as in ~\cite{GU22}, even when $R$ is not a field, there is a way to define  $\mathcal{L} \subset \mathcal{K} \subset \Omega \otimes_k R$. 
For a general $R$, reduced $k$-algebra of finite type, one may use the previous paragraph and Exercise X.16 of \cite{Lang}  to conclude that $\mathcal{L}^\#$ is a totally isotropic rank $d$ local direct summand of $\Omega \otimes_{k} R$, and $\mathcal{K}$ is a rank $d+1$ local direct summand. There is a unique totally isotropic rank $d$ local direct summand of $\Omega \otimes_k R$ contained in $\mathcal{K}$ that is not equal to   $ \mathcal{L}^\#$. Call this local direct summand $\mathcal{L}$. This defines the map:
$$\mathcal{N}_{\Lambda,0}(R) \rightarrow \mathrm{OGr}(d)(R)$$
$$(G, \rho) \mapsto \mathcal{L}$$
for any reduced $k$-algebra $R$ of finite type. As $\mathcal{N}_{\Lambda,0}$ is reduced and locally of finite type, this is enough to define a morphism of schemes:
$$\mathcal{N}_{\Lambda,0} \rightarrow \mathrm{OGr}(d).$$

Further, as  $\mathcal{X}_{\Lambda} \subset \mathrm{OGr}(d)$ is closed, and on field-valued points this morphism has image in $\mathcal{X}_{\Lambda}$, we have defined a morphism of schemes:
$$\mathcal{N}_{\Lambda,0} \rightarrow \mathcal{X}_{\Lambda}$$
inducing the map of Corollary ~\ref{mapsum} on $k'$-points.

\end{proof}

\subsection{Computing the Image}\label{Computingimage}

In the previous section, we showed that there is a morphism:
$$\mathcal{N}_{\Lambda,0} \rightarrow \mathcal{X}_{\Lambda}$$
inducing the map on $k'$-points $\mathcal{N}_{\Lambda,0}(k') \longrightarrow \mathcal{X}_{\Lambda}(k')$ described in Corollary ~\ref{mapsum}, for every field $k'$ over $k$. The goal of this section is to show that the morphism constructed in the previous section defines an isomorphism from $\mathcal{N}_{\Lambda,0}$ to $\mathcal{X}^+_{\Lambda}$.

\begin{proposition}\label{swaps}
For any vertex lattice $\Lambda$, the map $L \mapsto \overline{\Phi}(L)$ induces a map from $\mathcal{S}_{\Lambda}(k')$ to itself, which takes $\mathcal{S}^+_{\Lambda}(k')$  to $\mathcal{S}^-_{\Lambda}(k')$ and takes $\mathcal{S}^-_{\Lambda}(k')$  to $\mathcal{S}^+_{\Lambda}(k')$. Over $k$, this gives a bijection from $\mathcal{S}_{\Lambda}(k)$ to itself.

\end{proposition}

The proof of this proposition will require a lemma.

\begin{lemma}\label{lotsofords}
Let $C \subset \N'$ be a $W'$-lattice. The following two conditions are equivalent:
\begin{enumerate}
\item{ $$pC \subset \overline{F}(C) \subset C$$
$$\text{where   } \dim_{k'}(C/\overline{F}(C) ) = \dim_{k'}( \overline{F}(C)/pc) = 2$$ } 

\item{$$(\bigwedge_{W'}^2 C + \frac{1}{p} \bigwedge_{W'}^2 \overline{F}(C) )/ \bigwedge_{W'}^2 C \quad \text{has length 1}.$$}
\end{enumerate}
\end{lemma}

\begin{proof}
Let $C \subset \N'$ be a $W'$-lattice, and assume that:
$$pC \subset \overline{F}(C) \subset C$$
where  $\dim_{k'}(C/ \overline{F} (C) ) = \dim_{k'}( \overline{F} (C)/pc) = 2$.
Then there is a basis $\{c_i\}_1^4$ of $C$ as a free $W'$-module such that, without loss of generality, $\overline{F}(C)$ has free $W'$-basis  $\{pc_1, pc_2, c_3, c_4 \}$. Using this basis, it is easy to see that $(\bigwedge_{W'}^2 C + \frac{1}{p} \bigwedge_{W'}^2 \overline{F} (C) )/ \bigwedge_{W'}^2 C$ has length 1.


On the other hand,  assume that $(\bigwedge_{W'}^2 C + \frac{1}{p} \bigwedge_{W'}^2 \overline{F} (C) )/ \bigwedge_{W'}^2 C$ has length 1. Then there is a free basis $\{c_i\}_1^4$ for  $C$ as a $W'$-module such that $\{d_i c_i \}$ is a free basis for $\overline{F} (C)$, for some $d_i \in W'_{\Q}$. Without loss of generality, the condition that $(\bigwedge_{W'}^2 C + \frac{1}{p} \bigwedge_{W'}^2 \overline{F} (C) )/ \bigwedge_{W'}^2 C$ has length 1 produces the conditions:
$$\ord_p(c_1) + \ord_p(c_2) = 0$$
$$\ord_p(c_i) + \ord_p(c_j) \leq 1 \quad \text{for any pair   } i < j, \ \ (i,j) \neq (1,2).$$

Because $\Phi$ is slope-zero, we also have the condition:
$$3\ord_p(c_1) + 3\ord_p(c_2) + 3\ord_p(c_3) + 3\ord_p(c_4) = 6.$$

Combining the above equations, we can conclude that:
$$pC \subset \overline{F} (C) \subset C$$
where $\dim_{k'}(C/\overline{F} (C) ) = \dim_{k'}( \overline{F} (C)/pc) = 2$.

\end{proof}

Now we can prove the proposition.

\begin{proof}(Of Proposition ~\ref{swaps}) Let $\Lambda$ be a vertex lattice, and let $L \in \mathcal{S}_{\Lambda}(k')$. As $\Phi$ is slope zero and $[\Phi(x), \Phi(y)] = [x,y]^\sigma$, the self-duality of $L$ implies the self-duality of $\overline{\Phi}(L)$. Because $\Lambda \subset V^{\Phi}$, the fact that $L$ contains  $\Lambda^\vee$ implies that $\overline{\Phi}(L)$ contains $\Lambda^\vee$. So, to show that action by $\Phi$ gives a map from $\mathcal{S}_{\Lambda}(k')$ to itself, we only need to show that, if $L$ is a special lattice, then $\overline{\Phi}(L)$ meets the length 1 condition.

By Lemma ~\ref{twotypesoflattices}, we have two cases: $L = \bigwedge_{W'}^2 A$ for $A \subset \N'$ a $W'$-lattice, or $L = \frac{1}{p} \bigwedge_{W'}^2 B$ for $B \subset \N'$ a $W'$-lattice. Consider the first case, so $L = \bigwedge_{W'}^2 A$ for $A \subset \N'$ a $W'$-lattice. Using  Lemma ~\ref{lotsofords}, the length one condition for $L$:
$$(\bigwedge_{W'}^2 A + \frac{1}{p} \bigwedge_{W'}^2 \overline{F}(A)  )/ \bigwedge_{W'}^2 A \quad \text{ has length 1}$$
implies that:
$$p\overline{F}(A) \subset \overline{F}^2(A) \subset \overline{F}(A)$$
$$\text{with    } \dim_{k'}( \overline{F}(A)/\overline{F}^2(A) ) = \dim_{k'}(\overline{F}^2(A)/p\overline{F}(A) ) = 2$$
which implies that:
$$(\frac{1}{p} \bigwedge_{W'}^2 \overline{F}(A) + \frac{1}{p^2} \bigwedge_{W'}^2 \overline{F}^2(A) )/ \frac{1}{p} \bigwedge_{W'}^2 \overline{F}(A) \quad  \text{has length 1}.$$
This is precisely the length 1 condition for $\overline{\Phi}(L)$. The other case, where $L$ is of the form $\frac{1}{p} \bigwedge_{W'}^2 B$, similarly follows from two applications of  Lemma ~\ref{lotsofords}.

So, the map $L \mapsto \overline{\Phi}(L)$ gives a map from $\mathcal{S}_{\Lambda}(k')$ to itself, for any vertex lattice $\Lambda$. The fact that $L \mapsto \overline{\Phi}(L)$ takes $\mathcal{S}^{\pm}_{\Lambda}(k')$ to $\mathcal{S}^{\mp}(k')$  is clear from the definition of $\Phi$: for any $x \wedge y \in \bigwedge_{W'_{\Q}}^2 \N'$, by definition $\Phi(x \wedge y) = \frac{1}{p} \Phi(x) \wedge \Phi(y)$.

In the case that $k' = k$ is algebraically closed, one can use the fact that $F$ is a bijection from $\N$ to itself to show that action by $\Phi$ defines a bijection from $\mathcal{S}_{\Lambda}(k)$ to itself.

\end{proof}

\begin{proposition}\label{maptoonetype2}
Let $\Lambda$ be a vertex lattice. After possibly relabelling the components $\mathcal{X}_{\Lambda}^+$ and $\mathcal{X}_{\Lambda}^-$, the morphism:
$$\mathcal{N}_{\Lambda,0} \rightarrow \mathcal{X}_{\Lambda}$$ 
factors through an isomorphism over $k$: 
$$\mathcal{N}_{\Lambda,0} \rightarrow \mathcal{X}^+_{\Lambda}.$$
\end{proposition}

\begin{proof} We will begin by showing that the morphism $\mathcal{N}_{\Lambda,0} \rightarrow \mathcal{X}_{\Lambda}$ factors through $\mathcal{X}_{\Lambda}^+$. First consider the case where the vertex lattice $\Lambda$ is of type 2. By Proposition ~\ref{Xispt}, $\mathcal{X}_{\Lambda}(k)$ consists of two points, and each of $\mathcal{X}_{\Lambda}^{\pm}(k)$ consists of a single point. It follows from Proposition ~\ref{swaps} that each of   $\mathcal{S}_{\Lambda}^{\pm}(k)$ is also a single point.

Since $\mathcal{N}_{\Lambda,0}(k)$ is in bijection with $\mathcal{S}^+_{\Lambda}(k)$, both $\mathcal{N}_{\Lambda,0}(k)$ and its image in $\mathcal{X}_{\Lambda}(k)$ is a single point. After possibly relabelling the two components, the morphism $\mathcal{N}_{\Lambda,0} \rightarrow \mathcal{X}_{\Lambda}$ induces a map: 
$$\mathcal{N}_{\Lambda,0}(k) \rightarrow \mathcal{X}^+_{\Lambda}(k).$$ 
As the algebraically closed field $k$ is arbitrary, $\mathcal{N}_{\Lambda,0}$ is reduced, and $\mathcal{X}^+_{\Lambda}$ is closed in $\mathcal{X}_{\Lambda}$, the morphism $\mathcal{N}_{\Lambda,0} \rightarrow \mathcal{X}_{\Lambda}$ must factor through $\mathcal{X}_{\Lambda}^+$.

Now, consider the case that $\Lambda$ is of type 4.  By Proposition ~\ref{type4P1}, $\mathcal{X}_{\Lambda} = \mathcal{X}^+_{\Lambda} \sqcup \mathcal{X}^-_{\Lambda}$ is a disjoint union of two $\mathbb{P}^1$s. As $\mathcal{X}_{\Lambda}(k)$ is in bijection with $\mathcal{S}_{\Lambda}(k)$, the set $\mathcal{S}_{\Lambda}(k)$ is infinite, and as applying $\Phi$ induces a bijection from $\mathcal{S}_{\Lambda}^+(k)$ to $\mathcal{S}_{\Lambda}^-(k)$, both subsets must be infinite, and neither is equal to all of $\mathcal{S}_{\Lambda}(k)$. 

Since $\mathcal{N}_{\Lambda,0}(k)$ is in bijection with $\mathcal{S}^+_{\Lambda}(k)$, the image of $\mathcal{N}_{\Lambda,0}(k)$ in $\mathcal{X}_{\Lambda}(k)$ is infinite, but not all of $\mathcal{X}_{\Lambda}(k)$. Further, by Proposition ~\ref{projective}, $\mathcal{N}_{\Lambda,0}$ is proper. So the image of $\mathcal{N}_{\Lambda,0}(k)$ in each component of $\mathcal{X}_{\Lambda} = \mathcal{X}^+_{\Lambda} \sqcup \mathcal{X}^+_{\Lambda}$ must either be all of the $k$-points in that component, a finite set of points, or empty.

In summary, after possibly relabelling the two components, the only remaining possibilities for the image of $\mathcal{N}_{\Lambda,0}(k)$ in $\mathcal{X}_{\Lambda}(k)$ are:

\begin{enumerate}
\item{All of $\mathcal{X}^+_{\Lambda}(k)$, and no points of $\mathcal{X}^-_{\Lambda}(k)$ }
\item{All of $\mathcal{X}^+_{\Lambda}(k)$, and finitely many points of $\mathcal{X}^-_{\Lambda}(k)$.}
\end{enumerate}

For the sake of contradiction, assume that the second possibility holds. Let $\mathcal{L} \in \mathcal{X}^-_{\Lambda}(k)$ be in the image of $\mathcal{N}_{\Lambda,0}(k)$. Because $\mathcal{N}_{\Lambda,0}(k)$ is in bijection with $\mathcal{S}^+_{\Lambda}(k)$, there must be some very special lattice $L$ mapping to $\mathcal{L}$. Since applying $\Phi$ commutes with the map $\mathcal{S}_{\Lambda}(k) \rightarrow \mathcal{X}_{\Lambda}(k)$, the special lattice $\Phi(L)$ maps to $\Phi(\mathcal{L})$.

Recall that applying $\Phi$ interchanges both the pair of sets $\mathcal{X}^+_{\Lambda}(k)$ and $\mathcal{X}^-_{\Lambda}(k)$, and the pair of sets $\mathcal{S}^+_{\Lambda}(k)$ and $\mathcal{S}^-_{\Lambda}(k)$. Since  $\mathcal{L} \in \mathcal{X}^-_{\Lambda}(k)$, $\overline{\Phi}(\mathcal{L}) \in  \mathcal{X}^+_{\Lambda}(k)$. By the assumption that possibility two holds, there must be some very special lattice $K \in \mathcal{S}^+(k)$ mapping to $\Phi(\mathcal{L})$. But the unique special lattice mapping to $\Phi(\mathcal{L})$ is $\Phi(L)$, which is in $\mathcal{S}^-(k)$, not $\mathcal{S}^+(k)$. This is a contradiction.

Therefore, the first situation must hold, and the morphism $\mathcal{N}_{\Lambda,0} \rightarrow \mathcal{X}_{\Lambda}$ induces a map:
$$\mathcal{N}_{\Lambda,0}(k) \rightarrow \mathcal{X}_{\Lambda}^+(k).$$
Then,  because $\mathcal{N}_{\Lambda,0}$ is reduced and $\mathcal{X}^+_{\Lambda}$ is closed in $\mathcal{X}_{\Lambda}$, the morphism $\mathcal{N}_{\Lambda,0} \rightarrow \mathcal{X}_{\Lambda}$ must factor through $\mathcal{X}_{\Lambda}^+$.

Now, whether $\Lambda$ is of type 2 or type 4, we have a morphism of schemes:
$$\mathcal{N}_{\Lambda,0} \rightarrow \mathcal{X}^+_{\Lambda}$$ 
which, for any field $k'$ over $k$, induces the map on $k'$ points $\mathcal{N}_{\Lambda, 0}(k') \longrightarrow \mathcal{X}^+_{\Lambda}(k')$ described in Corollary ~\ref{mapsum}.

This map is in fact a bijection: the map $\mathcal{N}_{\Lambda,0}(k') \rightarrow \mathcal{X}^+_{\Lambda}(k')$ factors through the bijection between  $\mathcal{N}_{\Lambda,0}(k')$ and  $\mathcal{S}_{\Lambda}^+(k')$ so the bijection described in Proposition ~\ref{SX}:
$$\mathcal{S}_{\Lambda}(k') \longleftrightarrow \mathcal{X}_{\Lambda}(k')$$
$$L \mapsto \mathcal{L}$$
restricts to an injection: $\mathcal{S}^+_{\Lambda}(k') \rightarrow \mathcal{X}^+_{\Lambda}(k').$

As the map $L \mapsto \mathcal{L}$ is $\Phi$-equivariant, and by Proposition ~\ref{swaps} action by $\Phi$ interchanges both the subsets $\mathcal{S}^{\pm}_{\Lambda}(k')$ of $\mathcal{S}_{\Lambda}(k')$ and the subsets $\mathcal{X}^{\pm}_{\Lambda}(k')$ of $\mathcal{X}_{\Lambda}(k')$, the map $\mathcal{S}^+_{\Lambda}(k') \rightarrow \mathcal{X}^+_{\Lambda}(k')$ must be a bijection. As $\mathcal{N}_{\Lambda,0}(k')$ is in bijection with $\mathcal{S}^+_{\Lambda}(k')$, we have a morphism of schemes over $k$:
$$\mathcal{N}_{\Lambda,0} \rightarrow \mathcal{X}^+_{\Lambda}$$
inducing a bijection on $k'$-points for any field $k'$ over $k$.

It follows that the morphism $\mathcal{N}_{\Lambda,0} \rightarrow \mathcal{X}_{\Lambda}$ is quasi-finite and birational. By  Proposition ~\ref{projective}, it is also proper. Since $\mathcal{X}_{\Lambda}$ is normal, Zariski's main theorem implies that this morphism is an isomorphism.

\end{proof}

\section{Summary of the $\GL_4$ Rapoport-Zink Space and Further Observations}

\subsection{Intersection Behavior}\label{intersectionsection}

The conclusion of the previous section was that each $\mathcal{N}_{\Lambda,0}$ is isomorphic over $k$ to either a single point or to $\mathbb{P}^1$. We have that: 
$$\mathcal{N}_{0,red} = \bigcup_{\Lambda} \mathcal{N}_{\Lambda,0}$$
(where the index set is over all vertex lattices) and again this is \emph{not} a disjoint union:  $\mathcal{N}_{\Lambda_1} \subset \mathcal{N}_{\Lambda_2}$ whenever $\Lambda_1 \subset \Lambda_2$. In fact, from ~\cite{RTW} Proposition 4.3, we have:

$$
\mathcal{N}_{\Lambda_1} \cap \mathcal{N}_{\Lambda_2} =
\begin{cases}
\mathcal{N}_{\Lambda_1 \cap \Lambda_2} \quad \hfill \text{   if $\Lambda_1 \cap \Lambda_2$ is a vertex lattice} \\
\emptyset \hfill \text{otherwise.}
\end{cases}
$$

In the statement above, the right-hand side is by definition reduced. The left-hand side, defined to be the scheme-theoretic intersection, is reduced following a result of Li ~\cite{Li}.

It follows immediately from the definition of the type of a vertex lattice that the only nontrivial intersection behavior occurs when $\Lambda_1$ and $\Lambda_2$ are both vertex lattices of type 4, and then their intersection $\Lambda_1 \cap \Lambda_2$ must necessarily be of type 2.

The two natural questions are: on a single $\mathbb{P}^1$, how many intersection points are there? How many $\mathbb{P}^1$'s pass through a single intersection point? In terms of vertex lattices, these questions reduce to: how many type 2 vertex lattices are contained in a fixed vertex lattice of type 4? How many vertex lattices of type 4 contain a fixed vertex lattice of type 2? These questions are answered by the lemmas below.

\begin{lemma}\label{int1}
Let $\Lambda_b \subset V^{\Phi}$ be a fixed vertex lattice of type 4. There are exactly $p^2 +1 $ distinct vertex lattices $\Lambda_a$ of type 2 such that:
$$\Lambda_a \subset \Lambda_b.$$
\end{lemma}

\begin{proof}

Let $\Lambda_b$ be a fixed vertex lattice of type 4. Then we have the following chain condition:
$$p \Lambda_b \subset \Lambda_b^\vee \subset \Lambda_b \subset p^{-1} \Lambda_b^\vee.$$

Let $V_b = \Lambda_b / \Lambda_b^\vee$, which is of dimension 4 over $\F_p$. If $Q(x)$ denotes the quadratic form associated with the bilinear form $[\cdot, \cdot]$ on $V^{\Phi}$, give $V_b$ the quadratic form $q(x) = p Q(x) \mod p$. Note that this is well-defined, because $\Lambda_b \subset p^{-1} \Lambda_b^\vee$, and is also nondegerate. Further, because $V^\Phi$ does not have any self-dual lattices, $V_b$ is nonsplit. 

Let $\Lambda_a$ be a vertex lattice of type 2 contained in $\Lambda_b$. Then we have the following chain condition:
$$p \Lambda_a \subset p\Lambda_b \subset \Lambda_b^\vee \subset \Lambda_a^\vee \subset \Lambda_a \subset \Lambda_b$$
and so (by comparing dimensions) $\Lambda_a^\vee$ defines an isotropic line in $V_b$.

On the other hand, let $L \subset V_b$ be an isotropic line. We have the chain condition: 
$$\Lambda_b^\vee \subset L \subset L^\vee \subset \Lambda_b.$$ 
Define $\Lambda_a = L^\vee$. Then we have the chain condition:
$$p \Lambda_a \subset p\Lambda_b \subset \Lambda_b^\vee \subset \Lambda_a^\vee \subset \Lambda_a \subset \Lambda_b$$
and so (by comparing dimensions) $\Lambda_a$ would be a vertex lattice of type 2 contained in $\Lambda_b$.

Therefore, vertex lattices of type 2 contained in $\Lambda_b$ are in bijection with isotropic lines in $V_b$. Since $V_b$ is a nondegenerate, nonsplit quadratic space of dimension 4 over $\F_p$, there are $p^2 +1$ isotropic lines in $V_b$, so there are exactly $p^2 + 1$ vertex lattices of type 2 contained in a fixed vertex lattice of type 4.

\end{proof}

\begin{lemma}\label{int2}
Let $\Lambda_a \subset V^{\Phi}$ be a fixed vertex lattice of type 2. There are exactly $p^2 +1 $ distinct vertex lattices $\Lambda_b$ of type 4 such that:
$$\Lambda_a \subset \Lambda_b.$$
\end{lemma}

\begin{proof}

Let $\Lambda_a$ be a fixed vertex lattice of type 2. Then we have the following chain condition:
$$ p \Lambda_a \subset \Lambda_a^\vee \subset \Lambda_a.$$

Let $Q(x)$ denote the quadratic form associated with the bilinear form $[\cdot, \cdot]$ on $V^{\Phi}$. Because $\Lambda^\vee_a \subset \Lambda_a$, the quadratic form on $V^\Phi$ is $\Z_p$-valued on $\Lambda^\vee_a$, so we can define $V_a = \Lambda_a^\vee/ p \Lambda_a$ with the induced quadratic form $Q(x) \mod p$. Note that $V_a$ is a nondegenerate quadratic space of dimension 4 over $\F_p$ with this form, and is nonsplit, because $V^\Phi$ does not have any vertex lattices of type 6.

Then, similar to the previous argument, vertex lattices of type 4 containing $\Lambda_a$ are in bijection with isotropic lines in   $V_a$. As $V_a$ is a nondegenerate, nonsplit quadratic space over $\F_p$ of dimension 4, there are $p^2 +1$ such lines. Therefore, there are exactly  $p^2 +1$   distinct vertex lattices of type 4 containing $\Lambda_a$.

\end{proof}

\subsection{Superspecial points}

A $k$-point of $\mathcal{N}$ is called \emph{superspecial} if the corresponding $p$-divisible group is isomorphic to (not only isogenous to) the $p$-divisible group of a product of supersingular elliptic curves. (A $k'$-point will also be called superspecial if the corresponding $p$-divisible group is geometrically isomorphic to the $p$-divisible group of a product of supersingular elliptic curves.) As we've seen, $\mathcal{N}_{0,red}$ has irreducible components all isomorphic to projective lines,  indexed by vertex lattices of type 4, with intersection points indexed by vertex lattices  of type 2. These intersection points are exactly the superspecial points:

\begin{proposition}
A $k$-point of $\mathcal{N}_{red}$ lies at the intersection of multiple irreducible connected components if and only if it is a superspecial point.
\end{proposition}

\begin{proof}
Note that the isomorphisms given in Proposition ~\ref{allheights} send superspecial points to superspecial points, so it is enough to consider $\mathcal{N}_{0,red}$. Given a superspecial point $(G, \rho) \in \mathcal{N}_{0,red}(k)$, let $M \subset \N$ be the corresponding Dieudonn\'e lattice. Because $G$ is superspecial, $FM = VM$. Then the very special lattice $L = \frac{1}{p} \bigwedge_{W}^2 M_1$ satisfies:
$$L = \Phi^2(L).$$
By Proposition ~\ref{vollaardprop}, $\Lambda_L$ is then a vertex lattice of type 2. As $L \in \mathcal{S}^+_{\Lambda_L}(k)$, we have $(G, \rho) \in \mathcal{N}_{\Lambda_L,0}(k)$, where $\mathcal{N}_{\Lambda_L,0}$ is the intersection of $p^2 + 1$ irreducible components isomorphic to $\mathbb{P}^1$. 

On the other hand, let $(G, \rho) \in \mathcal{N}_{0,red}(k)$ be an intersection point, with corresponding Dieudonn\'e lattice $M$. Then there is a vertex lattice $\Lambda$ of type 2 such that the  very special lattice $L = \frac{1}{p} \bigwedge_{W}^2 M_1$ is contained in $\mathcal{S}^+_{\Lambda}(k)$. Because $\Lambda$ is a vertex lattice of type 2, $\mathcal{S}_{\Lambda}(k)$ consists of two points, interchanged by Frobenius. So,
$$L = \Phi^2(L).$$
From this it follows that $FM = VM$, and so $(G, \rho)$ is a superspecial point.
\end{proof}

\subsection{Ekedahl-Oort Stratification}

The $\GL_4$ Rapoport-Zink space $\mathcal{N}$ is a moduli space of $p$-divisible groups. There is a natural stratification of $\mathcal{N}_{red}$ (the reduced $k$-scheme underling $\mathcal{N}$) by the following rule: two points $(G_1, \rho_1)$ and $(G_2, \rho_2)$ are in the same stratum if and only if $G_1[p]$ and $G_2[p]$ are geometrically isomorphic. This is called the Ekedahl-Oort stratification (see Oort ~\cite{EO}). In this section, we will describe the Ekedahl-Oort stratification of $\mathcal{N}_{red}$. 

If $G$ is a $p$-divisible group over $k$,  then $G[p]$ is a finite, commutative group scheme over $k$ which is killed by $p$. By Dieudonn\'e theory, there is an equivalence of categories between the category of finite, commutative group schemes $X$ over $k$ that are killed by $p$ and the category of triples $(M, F, V)$ where:
\begin{itemize}
\item{$M$ is a finite dimensional $k$-vector space}
\item{$F: M \rightarrow M$ is a $\sigma$-linear map}
\item{$V: M \rightarrow M$ is a $\sigma^{-1}$-linear map}
\end{itemize}
with the property that $F \circ V = 0 = V \circ F$.  (See, for example, ~\cite{Demazure} for details.) 


A finite, commutative group scheme over $k$ that is killed by $p$ and that arises as the $p$-torsion of a $p$-divisible group over $k$ is called a $BT_1$ group scheme. Following a result of Illusie ~\cite{Ill}, the category of $BT_1$ group schemes $X$ over $k$ is equivalent to the category of triples $(M, F, V)$ described above, with the additional condition that:
$$\mathrm{Ker}(F) = \mathrm{Im}(V) \quad \quad \text{and} \quad \quad \mathrm{Ker}(V) = \mathrm{Im}(F).$$

 The first combinatorial analysis of such triples was done in an unpublished work of Kraft ~\cite{Kraft} in 1975. His results are summarized by Moonen in ~\cite{Moo} and Oort in ~\cite{EO}  and are briefly recalled here:
 
 Consider the set of finite words in the symbols $F$ and $V$. Two words will be considered the same if they are equal under some cyclic rotation (e.g. VVF = VFV = FVV) and a word will be called simple if it is not periodic under this action of period larger than one. Let $W$ be the set of finite, simple words up to cyclic rotation. The set $W$ classifies indecomposable $BT_1$ group schemes according to the following construction:
 
 Given a word $w = S_1S_2\cdots S_n$, let $M_w$ be a $k$ vector space with basis $\{e_i\}_1^n$. The basis vector $e_1$ may also be denoted by $e_{n+1}$. If $S_i = F$, let $F(e_i) = e_{i+1}$ and $V(e_{i+1}) = 0$. If $S_i = V$, let $V(e_{i+1}) = e_i$ and $F(e_i) = 0$. This gives a triple $(M_w, F, V)$, which determines a $BT_1$ group scheme $X_w$ over $k$.

\begin{lemma}
Let $X$ be a  $BT_1$ group scheme over $k$ that arises as the $p$-torsion of a supersingular $p$-divisible group $G$ of height 4 and dimension 2. With the notation described above, $X$ is isomorphic to either $(X_{FV})^2$ or $X_{FFVV}$.
\end{lemma}

\begin{proof}
As discussed above, the category of $BT_1$ group schemes over $k$ is equivalent to the category of triples $(M, F, V)$ with some constraints. A triple $(M,F,V)$ corresponds to a $BT_1$ group scheme $X$ which is the $p$-torsion of a supersingular $p$-divisible group $G$ of height 4 and dimension 2 if and only if the following additional conditions are also met:
\begin{enumerate}
\item{$\dim_k(M) = 4$}
\item{$\dim_k(V(M)) = 2 = \dim_k(F(M))$}
\item{$F$ and $V$ act nilpotently on $M$}
\end{enumerate}
The dimension conditions (1) and (2) on $M$ are equivalent to the requirement that $G$ is of height 4 and dimension 2. The third condition ensures that the isocrystal of $G$ will have all slopes equal to $\frac{1}{2}$.

If $X$ is indecomposble, and we require $M$ to be of dimension 4, then by the classification of Kraft described above $G$ is isomorphic to one of:
$$X_{FFFF}, \quad X_{FFFV}, \quad X_{FFVV}, \quad X_{VVVF}, \quad \text{or} \quad X_{VVVV}.$$

Of these, only $X_{FFVV}$ meets conditions (2) and (3).

If $X$ is decomposable as a $BT_1$ group scheme, and we require $M$ to be of dimension 4, then $X$ must be a product of one of the following forms:
\begin{itemize}
\item{$X_1 \times X_2$, where $X_1$ is rank $p$ and $X_2$ is rank $p^3$}
\item{$X_1 \times X_2 \times X_3 \times X_4$ where all $X_i$ are rank $p$}
\item{$X_1 \times X_2$ where both are rank $p^2$}
\end{itemize}

The first two cases may be ruled about by observing that the only $BT_1$ group schemes of rank $p$ are $X_F = \underline{\Z / p\Z }$ and $X_V = \mu_p$, and neither of these have the property that both $F$ and $V$ act nilpotently. The only $BT_1$ group schemes of rank $p^2$ are $X_{FF}$, $X_{FV}$, and $X_{VV}$. Of these, only $X_{FV}$ has the property that $F$ and $V$ act nilpotently. So the only remaining possibility is that $X \cong (X_{FV})^2$.
\end{proof}

It is worth noting that $X_{FFVV}$ and $X_{FV}$ have another common description. The group $X_{FV}$ can be described as the $p$-torsion of a supersingular elliptic curve over $k$, and is sometimes denoted as $I_{1,1}$. The group $X_{FFVV}$ can be described as the $p$-torsion of a supersingular, but not superspecial, abelian surface, and is sometimes denoted $I_{2,1}$. See Pries ~\cite{Pries} for details.

In Sections 3.1-3.3 of ~\cite{EO}, Oort defines the Ekedahl-Oort strata of a Siegel modular variety and shows that they are locally closed. By similar reasoning, if we let $\mathcal{N}_{FFVV,red}$ denote the locus in $\mathcal{N}_{red}$ where the $p$-torsion groups are geometrically isomorphic to  $X_{FFVV}$ and let  $\mathcal{N}_{FVFV,red}$ be the locus in $\mathcal{N}_{red}$ where the $p$-torsion groups are geometrically isomorphic to  $(X_{FV})^2$, then $\mathcal{N}_{FVFV,red}$ defines a closed subscheme of $\mathcal{N}_{red}$, and its complement $\mathcal{N}_{FFVV,red}$ is locally closed. These are the two Ekedahl-Oort strata of $\mathcal{N}_{red}$.

\begin{proposition}\label{EO}
The Ekedahl-Oort stratum $\mathcal{N}_{FVFV,red}$ is exactly the collection of superspecial points of $\mathcal{N}_{red}$. The Ekedahl-Oort stratum $\mathcal{N}_{FFVV,red}$ is then the complement of the  superspecial points.
\end{proposition}

\begin{proof}
The key observation is that $(\mathcal{X}_{FV})^2$ is a \emph{minimal} $BT_1$ group scheme, in the language of Oort ~\cite{Min}.  Let $(G, \rho)$ be a $k'$-point of $\mathcal{N}_{red}$ that is superspecial. Then, $G$ is geometrically isomorphic to the $p$-divisible group of a product of two supersingular elliptic curves, so $G[p]$ is geometrically isomorphic to $(X_{FV})^2$, and $(G,\rho)$ is in the Ekedahl-Oort stratum $\mathcal{N}_{FVFV,red}$.

On the other hand, let $(G, \rho)$ be a $k'$-point of $\mathcal{N}_{FVFV,red}$. Then, $G[p]$ is geometrically isomorphic to $(X_{FV})^2$. By Oort ~\cite{Min}, because  $(X_{FV})^2$ is minimal, $G$ is geometrically isomorphic to the $p$-divisible group of a product of two supersingular elliptic curves. So, $(G,\rho)$ defines a superspecial point of $\mathcal{N}_{red}$.
\end{proof}



\subsection{Summary of Results}

We will now summarize our results on the geometry of the $\GL_4$ Rapoport-Zink space:

\begin{theorem}\label{gl4thm}
The $\GL_4$ Rapoport-Zink space $\mathcal{N}$ decomposes as a disjoint union:
$$\mathcal{N} = \bigsqcup_{i \in \Z} \mathcal{N}_i$$
where $\mathcal{N}_i$ is the locus of points $(G, \rho)$ where the quasi-isogeny $\rho$ is of height $i$. The components $\mathcal{N}_i$ are all isomorphic, as formal schemes over $\mathrm{Spf}(W )$.

Let $\mathcal{N}_{0,red}$ be the reduced $k$-scheme underlying $\mathcal{N}_0$. Then $\mathcal{N}_{0,red}$ is connected, and decomposes as:
$$\mathcal{N}_{0,red} = \bigcup_{\Lambda} \mathcal{N}_{\Lambda}$$
where the index set is the collection of all vertex lattices $\Lambda$ in $V^{\Phi}$ of type 4. These $\mathcal{N}_{\Lambda}$ are precisely the irreducible components of $\mathcal{N}_{0,red}$, and each is isomorphic, over $k$, to $\mathbb{P}^1$.

Two irreducible components are either disjoint or intersect in a single point.  Each  irreducible component contains $p^2 + 1$ intersection points, and each intersection point is the intersection of $p^2 + 1$ irreducible components. The intersection points are exactly the superspecial points of $\mathcal{N}_{0,red}$, and they are parametrized by the vertex lattices $\Lambda$ in $V^{\Phi}$ of type 2.

Further, $\mathcal{N}_{red}$ (the reduced $k$-scheme underlying $\mathcal{N}$) has two Ekedahl-Oort strata: one consisting of the superspecial points, and the other the complement of the superspecial points.
 \end{theorem}

\section{The $\GU(2,2)$ Rapoport-Zink Space}

The definition of the $\GU(2,2)$ Rapoport-Zink space relies on an imaginary quadratic field $E$. When the prime $p$ is split in $E$, the $\GU(2,2)$ Rapoport-Zink space decomposes into a disjoint union of copies of the $\GL_4$ Rapoport-Zink space. In this section, we will introduce the $\GU(2,2)$ Rapoport-Zink space, and explain this decomposition in the case that the prime $p$ is split. We will then use the results of the previous section to give a description of the $\GU(2,2)$ Rapoport-Zink space.

Recall that $k$ is an algebraically closed field of characteristic $p$, where $p \neq 2$,  $W = W(k)$ is the ring of Witt vectors of $k$, with fraction field denoted $W_{\Q}$, and  $\Nilp_W$ is the category of $W$-schemes on which $p$ is locally nilpotent. Let $E$ be an imaginary quadratic field in which the prime $p$ is split. The $GU(2,2)$ Rapoport-Zink space will parametrize isomorphism classes of supersingular $p$-divisible groups of height 8 and dimension 4, with an action of $\mathcal{O}_E \otimes_{\Z} \Z_p$, a principal polarization, and an isogeny to a fixed basepoint. So we will begin by constructing a particular choice of basepoint.

Let $\mathbb{G}$ continue to denote the basepoint for the $\GL_4$ Rapoport-Zink space. We will also continue to use the notation: $\underline{\mathbb{G}} = \mathbb{G} \times_k \mathbb{G}^*.$

Observe that, as $p$ is split in $E$, $\mathcal{O}_E \otimes_{\Z} \Z_p \cong \Z_p \times \Z_p$. Note that the nontrivial automorphism of $E \otimes_{\Q} \Q_p$ over $\Q_p$, denoted $\alpha \mapsto \overline{\alpha}$,  interchanges the two $\Z_p$ factors in $ \mathcal{O}_E \otimes_{\Z} \Z_p$. There is naturally an action of $\Z_p$ on both $\G$ and $\G^*$, so define:
$$\iota_{\underline{\G} }: \mathcal{O}_E \otimes_{\Z} \Z_p \rightarrow \End(\underline{\G})$$
by having $\mathcal{O}_E \otimes_{\Z} \Z_p  \cong \Z_p \times \Z_p$ act on $\G$ through the first factor of $\Z_p$ and act on $\G^*$ through the second factor of $\Z_p$. 

We also need a principal polarization of the basepoint. Let $a \mapsto e_a$ denote the canonical isomorphism $\mathbb{G} \rightarrow (\mathbb{G}^*)^*$. 

Define:
$$\lambda_{\underline{\G}}: \underline{\mathbb{G}} =  \G \times_k \G^* \rightarrow \G^* \times_k (\G^*)^* \cong \underline{\mathbb{G} }^*$$
$$(a,f) \mapsto (f, e_{a^{-1}}).$$
This is a principal polarization. 

The following hold by construction: as $\mathbb{G}$ is  supersingular of height 4 and dimension 2, we have that $\underline{\G}$ is supersingular of height 8 and dimension 4. Further, the action $\iota_{\underline{\G}}$ and polarization $\lambda_{\underline{\G}}$ satisfy the \emph{Rosati involution condition}:
$$\lambda_{\underline{\G}} \circ  \iota_{\underline{\G}}  (\overline{\alpha}) =  \iota_{\underline{\G}} (\alpha)^* \circ \lambda_{\underline{\G}}$$
for every $a \in \mathcal{O}_E \otimes_{\Z} \Z_p$. The action $\iota_{\underline{G}}$ also satisfies the \emph{signature $(2,2)$ condition}:
$$\det(T - \iota_{\underline{\G}}(a); \mathrm{Lie}(\underline{\G}) ) = (T - \psi_0(a))^2(T - \psi_1(a)) \in k[T]$$
for any $a \in \mathcal{O}_E \otimes_{\Z} \Z_p$, if $\psi_0$ and $\psi_1$ are the two $\Z_p$-morphisms of $\mathcal{O}_E \otimes_{\Z} \Z_p$ to $k$ (given by the maps $\Z_p \rightarrow \F_p$ on each $\Z_p$-factor).

Using the triple $(\underline{\G}, \iota_{\underline{\G}}, \lambda_{\underline{\G}})$ as a basepoint, we can define the $\GU(2,2)$ Rapoport-Zink space (when the prime $p$ is split in $E$). Define a functor:
$$\mathcal{N}_{GU(2,2)}: \mathrm{Nilp}_W \rightarrow Sets$$
that sends  a $W$-scheme $S$ on which $p$ is locally nilpotent to the set $\mathcal{N}_{GU(2,2)}(S)$ of quadruples $(G, \iota, \lambda, \rho)$, where:
\begin{itemize}
\item{$G$ is a supersingular $p$-divisible group over $S$ of height 8 and dimension 4}
\item{$\iota$ is an action of $\mathcal{O}_E \otimes_{\Z} \Z_p$ on $G$}
\item{$\lambda$ is a principal polarization of $G$}
\item{$\rho$ is an $\mathcal{O}_E \otimes_{\Z} \Z_p$-linear  quasi-isogeny from $G_{S_0}$ to $\underline{\G}_{S_0}$ }
 \end{itemize}
 Such that:
 \begin{itemize}
 \item{$\iota$ and $\lambda$ satisfy the \emph{Rosati involution condition} }
 \item{$\iota$ satisfies the  \emph{signature (2,2) condition}, which is formulated over $S$ by taking $\psi_0$ and $\psi_1$ to be the two embeddings of $\mathcal{O}_E \otimes_{Z} \Z_p$ into $W$, and considering the determinant in $\mathcal{O}_S[T]$ }
 \item{Locally on $S_0$, $\rho^*( \lambda_{\underline{\G}}  ) = c \lambda$ for some $c \in \Q_p^\times$. }
 \end{itemize}
In the above, $\mathcal{S}_0 = S \times_{\Spec(W) } \Spec(k)$.

Two quadruples $(G_1, \iota_1, \lambda_1, \rho_1)$ and $(G_2, \iota_2, \lambda_2, \rho_2)$ are isomorphic if there is an isomorphism  of $p$-divisible groups over $S$ from $G_1$ to $G_2$  that is $\mathcal{O}_E \otimes_{\Z} \Z_p$-linear, that carries $\rho_2$ to $\rho_1$, and such that $\lambda_2$ pulls back to a $\Z_p^\times$-multiple of $\lambda_1$. This definition of $\mathcal{N}_{GU(2,2)}$ does reference a particular choice of basepoint. However, any possible basepoint is isogenous (by an isogeny respecting to the extra structure) to the one constructed above, so this choice does not affect the geometry of the Rapoport-Zink space $\mathcal{N}_{GU(2,2)}$.

Given an $S$-point $(G, \iota, \lambda, \rho)$ of  $\mathcal{N}_{GU(2,2)}$ for which $S_0$ is connected, there is by definition some $c \in \Q_p^\times$ such that $\rho^*( \underline{\lambda}) = c \lambda$. This $c \in \Q_p^\times$ is \emph{not} an invariant of the isomorphism class of $(G, \iota, \lambda, \rho)$, but $\ord_p(c)$ \emph{is} an invariant. Then, $\mathcal{N}_{GU(2,2)}$ is represented by a formal scheme over $W$ that admits a decomposition into open and closed formal subschemes:
$$\mathcal{N}_{GU(2,2) } = \bigsqcup_{i \in \Z} \mathcal{N}^i_{GU(2,2)}$$
where $\mathcal{N}^i_{GU(2,2)}$ is the locus on which $\ord_p(c) = i$.

\begin{proposition}\label{label}
Let $\mathcal{N}_{GU(2,2)}$ be the $\GU(2,2)$ Rapoport Zink space, and assume that the prime $p$ is split in the imaginary quadratic field $E$. Let $\mathcal{N}$ be the $\GL_4$ Rapoport-Zink space. Then, for any integer $i \in \Z$, there is an isomorphism:
$$\mathcal{N}^i_{GU(2,2)} \cong \mathcal{N}.$$
\end{proposition}

\begin{proof}
Fix an integer $i$, and consider a quadruple $(G, \iota, \lambda, \rho) \in \mathcal{N}^i_{GU(2,2)}(S)$, for $S \in \mathrm{Nilp}_W$ such that $S_0$ is connected. The key observation is that, because $\mathcal{O}_E \otimes_{\Z} \Z_p \cong \Z_p \times \Z_p$ acts on $G$, we have:
$$G \cong \epsilon_0 G \times_S \epsilon_1 G,$$
where $\epsilon_0$ and $\epsilon_1$ in  $\mathcal{O}_E \otimes_{\Z} \Z_p$ correspond to $(1,0)$ and $(0,1)$ in $\Z_p \times \Z_p$, respectively. 
Because $\rho: G_{S_0} \rightarrow \underline{\G}_{S_0}$ is  $\mathcal{O}_E \otimes_{\Z} \Z_p$-linear, it restricts to a quasi-isogeny:
$$\rho_0: (\epsilon_0 G)_{S_0} \rightarrow \mathbb{G}_{S_0}.$$

Then, as $\lambda$ is a principal polarization satisfying the Rosati involution condition, it defines an isomorphism:
$$\lambda: \epsilon_1G \xrightarrow{\sim} (\epsilon_0 G)^*.$$
Using the fact that $\rho^*( \lambda_{\underline{\G}}   ) = c \lambda$, under the isomorphism $ \epsilon_1G \xrightarrow{\sim} (\epsilon_0 G)^*$, this quasi-isogeny can be considered on components as:
$$\rho : G_{S_0} \cong (\epsilon_0 G)_{S_0} \times_{S_0} (\epsilon_0 G)_{S_0}^* \xrightarrow{ (\rho_0, c (\rho_0^*)^{-1}) } \G_{S_0} \times_{S_0} \mathbb{G}_{S_0}^* = \underline{\G}_{S_0} .$$

Finally, because $\lambda$ is a polarization, we have:
$$\lambda: G \cong \epsilon_0 G \times_S (\epsilon_0 G)^* \rightarrow (\epsilon_0 G)^* \times_S ((\epsilon_0 G)^*)^* \cong G^*$$
$$(a, f) \mapsto (f, e_{a^{-1}}).$$
Note that because $G$ is supersingular of height 8 and dimension 4, $\epsilon_0 G$ must be supersingular of height 4 and dimension 2.

After these observations, it is easy to define the bijection between $\mathcal{N}^i_{GU(2,2)}(S)$ and $\mathcal{N}(S)$. It suffices to consider $S$ for which $S_0$ is connected. Given a  quadruple $(G, \iota, \lambda, \rho) \in \mathcal{N}^i_{GU(2,2)}(S)$, the corresponding element of $\mathcal{N}(S)$ is $(\epsilon_0 G, \rho_0).$

Given an element $(G, \rho) \in \mathcal{N}(S)$, we can construct the corresponding element of $\mathcal{N}^i_{GU(2,2)}(S)$ as follows. Let $\underline{G} = G \times_S G^*$, choose $c \in \Q_p^\times$ with $\ord_p(c) = i$, and define:
$$\underline{\rho}: \underline{G}_{S_0} = G_{S_0} \times_{S_0} G_{S_0}^* \xrightarrow{ (\rho, c (\rho^*)^{-1}) }  \G_{S_0} \times_{S_0} \mathbb{G}_{S_0}^* = \underline{\G}_{S_0}.$$
Define $\iota: \mathcal{O}_E \otimes_{Z} \Z_p \rightarrow \End(\underline{G})$ by acting on $G$ by the first factor of $\Z_p$ and acting on $G^*$ through the second factor. This action satisfies the signature (2,2) condition and causes $\underline{\rho}$ to be $ \mathcal{O}_E \otimes_{Z} \Z_p $-linear.  Finally, construct a principal polarization of $\underline{G}$ by:
$$\lambda: \underline{G} =  G \times_S G^* \rightarrow G^* \times_S (G^*)^* = \underline{G }^*$$
$$(a, f) \mapsto (f, e_{a^{-1}})$$
and note that  $\underline{\rho}^*( \lambda_{\underline{\G}}   ) = c \lambda$. Taken together, this defines an element $(\underline{G}, \iota, \lambda, \underline{\rho})$ of $\mathcal{N}^i_{GU(2,2)}(S)$. 

As the maps between $\mathcal{N}^i_{GU(2,2)}(S)$ and  $\mathcal{N}(S)$ just described are inverses of each other, and are functorial, this defines an isomorphism:
$$\mathcal{N}^i_{GU(2,2)} \cong \mathcal{N}.$$

\end{proof}

The isomorphism $\mathcal{N}^i_{GU(2,2)} \cong \mathcal{N}$ combined with the description in Proposition ~\ref{EO} of the Ekedahl-Oort stratification of $\mathcal{N}_{red}$ immediately produces a description of the Ekedahl-Oort stratification of $\mathcal{N}^i_{GU(2,2),red}$ and therefore of $\mathcal{N}_{GU(2,2),red}$. To explain this, we must first fix some notation for the relevant $BT_1$ group schemes.

Let $X_{1} = (X_{FV})^2 \times ((X_{FV})^2)^*$. (The group scheme $X_{FV}$ is self-dual, but it is convenient to define $X_{1}$ this way.) Let $\mathcal{O}_E \cong \Z_p \times \Z_p$ act on the factor $(X_{FV})^2$ by first $\Z_p$-factor and on the factor $((X_{FV})^2)^*$ through the second $\Z_p$-factor. Define a polarization of $X_{1}$ by:
$$\lambda: X_{1} =(X_{FV})^2 \times ((X_{FV})^2)^* \rightarrow ((X_{FV})^2)^* \times (((X_{FV})^2)^*)^* \cong X_{1}^*$$
$$(a,f) \mapsto (f, e_{a^{-1}} ).$$

Let $X_{2} = X_{FFVV} \times X^*_{FFVV}$.  Let $\mathcal{O}_E \cong \Z_p \times \Z_p$ act on the factor $X_{FFVV}$ through the first $\Z_p$-factor and act on the factor $X_{FFVV}^*$ through the second $\Z_p$-factor. Define a polarization of $X_{2}$ by:
$$\lambda: X_{2} =X_{FFVV} \times X^*_{FFVV} \rightarrow X^*_{FFVV} \times (X^*_{FFVV})^* \cong X_{2}^*$$
$$(a,f) \mapsto (f, e_{a^{-1}} ).$$

Now we can describe the Ekedahl-Oort stratification of $\mathcal{N}_{GU(2,2),red}$: two points are in the same stratum if the corresponding $p$-torsion subgroups, equipped with $\mathcal{O}_E$ action and polarization, are geometrically isomorphic. A $k$-point of the $\GU(2,2)$ Rapoport-Zink space is said to be \emph{superspecial} if the corresponding $p$-divisible group is isomorphic to the $p$-divisible group of a product of supersingular elliptic curves. 

Let $\mathcal{N}^i_{GU(2,2),red}(1)$ be the image of $\mathcal{N}_{FVFV,red}$ under the isomorphism of Proposition ~\ref{label} and let $\mathcal{N}_{GU(2,2),red}(1) = \sqcup_{i \in \Z} \mathcal{N}^i_{GU(2,2),red}(1) \subset \mathcal{N}_{GU(2,2),red}$. Define $\mathcal{N}_{GU(2,2),red}(2)$ similarly, with $\mathcal{N}_{FVFV,red}$ replaced by $\mathcal{N}_{FFVV,red}$.

\begin{proposition}
The reduced $k$-scheme $\mathcal{N}_{GU(2,2),red}$ underlying the $\GU(2,2)$ Rapoport-Zink space has two Ekedahl-Oort strata: $\mathcal{N}_{GU(2,2),red}(1)$ which coincides with the locus of superspecial points, and $\mathcal{N}_{GU(2,2),red}(2)$, the complement of the superspecial points.
\end{proposition}

\begin{proof}
This follows immediately from the construction of the isomorphism in Proposition ~\ref{label} and the description of the Ekedahl-Oort stratification of $\mathcal{N}_{red}$ in Proposition ~\ref{EO}: $\mathcal{N}_{GU(2,2),red}(1)$ is the locus on which the $p$-torsion subgroups are geometrically isomorphic to $X_1$, and $\mathcal{N}_{GU(2,2),red}(2)$ is the locus on which the $p$-torsion subgroups are geometrically isomorphic to $X_2$. 
\end{proof}

We can use our analysis of the $\GL_4$ Rapoport-Zink space and the above proposition to produce the following description of the $\GU(2,2)$ Rapoport-Zink space.

\begin{theorem}\label{RZ22}
Let $\mathcal{N}_{GU(2,2)}$ be the $\GU(2,2)$ Rapoport Zink space, and assume that the prime $p$ is split in the imaginary quadratic field $E$. The underlying reduced $k$-scheme $\mathcal{N}_{GU(2,2), red}$ has connected components naturally indexed by $\Z \times \Z$, and the connected components are all isomorphic. 

Every connected component of $\mathcal{N}_{GU(2,2), red}$  is isomorphic, over $k$, to a union of projective lines. Every pair of projective lines is either disjoint or intersects in a single point. Each projective line contains $p^2 + 1$ intersection points, and each intersection point is the intersection of $p^2 + 1$ projective lines. These intersection points are precisely the superspecial points.

Further, $\mathcal{N}_{GU(2,2),red}$ has two Ekedahl-Oort strata: one consisting of the superspecial points, and the other the complement of the superspecial points.
\end{theorem}

\begin{proof}
The $\GU(2,2)$ Rapoport-Zink space decomposes as:
$$\mathcal{N}_{GU(2,2) } = \bigsqcup_{i \in \Z} \mathcal{N}^i_{GU(2,2)}.$$

Using Proposition ~\ref{label} and Theorem ~\ref{gl4thm}, each $ \mathcal{N}^i_{GU(2,2),red}$ is isomorphic to $\mathcal{N}_{red}$, and $\mathcal{N}_{red}$ decomposes into connected components as $\mathcal{N}_{red} = \sqcup_{j \in \Z} \mathcal{N}_{j,red}$. So we have a decomposition of $ \mathcal{N}^i_{GU(2,2),red}$ into connected components as:
$$\mathcal{N}_{GU(2,2),red} =  \bigsqcup_{(i,j) \in \Z \times \Z } \mathcal{N}^{(i,j)}_{GU(2,2),red}$$
where $\mathcal{N}^{(i,j)}_{GU(2,2)}$ is the locus of points where the scalar $c \in \Q_p^\times$ from the condition $\rho^*(\underline{\lambda}) = c \lambda$ satisfies $\ord_p(c) = i$, and the restricted quasi-isogeny $\rho_0: \epsilon_0 G \rightarrow \mathbb{G}$ is of height $j$. By Proposition ~\ref{label}, each $\mathcal{N}^{(i,j)}_{GU(2,2)}$ is isomorphic to $\mathcal{N}_{0,red}$. The description of the irreducible components and intersection behavior of $\mathcal{N}_{GU(2,2),red}$ follows from the similar description for $\mathcal{N}_{0,red}$, and the fact that the intersection points are precisely the superspecial points is due to the fact that the isomorphism in Proposition ~\ref{label} respects superspecial points.

\end{proof}

\section{The $\GU(2,2)$ Shimura Variety}\label{SV}

We will now introduce the $\GU(2,2)$ Shimura variety and use the Rapoport-Zink uniformization theorem, combined with the results of the previous section, to describe its supersingular locus at a prime $p$ split in the relevant field.

Let $E$ continue to be an imaginary quadratic field, let $p$ continue to be an odd prime split in $E$, and let $\mathcal{O}$ be the integral closure of $\mathbb{Z}_{(p)}$ in $E$. Fix a free $\mathcal{O}$-module $V$ of rank 4,  with a perfect $\mathcal{O}$-valued Hermitian form of signature $(2,2)$. Let $G = \GU(V)$, fix a compact open subgroup $K^p$ of $\G(\A_f^p)$, and let $K_p = \G(\Z_p)$ and $K = K_pK^p$. 

For sufficiently small $K^p$, there is a smooth complex manifold $\mathcal{M}_K(\C)$ that is a moduli space of isomorphism classes of quadruples $(A, \iota, \lambda, \eta^p K^p)$, where:
\begin{itemize}
\item{ $A$ is an abelian variety over $\C$ of dimension 4}
\item{ $\iota: \mathcal{O} \rightarrow \mathrm{End}(A) \otimes_{\Z} \Z_{(p)}$ is an action of $\mathcal{O}$ on $A$, satisfying the \emph{signature $(2,2)$ condition}: For all $a \in \mathcal{O}$,
$$\det(T - \iota(a); \mathrm{Lie}(A) ) = (T - a)^2(T-\overline{a})^2.$$}
\item{$\lambda \in \Hom(A, A^*) \otimes_{\Z} \Z_{(p)}$ is a prime-to-$p$ quasi-polarization of $A$, which satisfies the \emph{Rosati involution condition}: For all $a \in \mathcal{O}$,
$$\lambda \circ \iota(\overline{a}) = \iota(a)^* \circ \lambda.$$}
\item{The level structure $\eta^pK^p$ is the $K^p$ orbit of an $\mathcal{O} \otimes \A_f^p$-linear isomorphism:
$$\eta^p: \widehat{T(A)}^p \otimes \A_f^p \rightarrow V \otimes \A_f^p$$
that respects the hermitian forms on either side, up to scaling by $(\A_f^p)^\times$.}
\end{itemize}

Two quadruples $(A_1, \iota_1, \lambda_1, \eta_1^p K^p)$ and $(A_2, \iota_2, \lambda_2, \eta_2^p K^p)$ are said to be isomorphic if there is an $\mathcal{O}$-linear quasi-isogeny in $\Hom(A_1, A_2) \otimes_{\Z} \Z_{(p)}$ of degree prime to $p$, respecting the level structures, such that $\lambda_2$ pulls back to a $\Z_{(p)}^\times$-multiple of $\lambda_1$.

For sufficiently small $K^p$, there is a scheme $\mathcal{M}_{K}$ over $\Z_{(p)}$,  such that, as suggested by the notation, $\mathcal{M}_K(\C)$ is the complex manifold described above, and $\mathcal{M}_K$ is the moduli space for the similarly-defined moduli problem over $\Z_{(p)}$ (see ~\cite{Ko} for details). This is (an integral model of) the $\GU(2,2)$ Shimura variety. It is smooth and of relative dimension 4 over $\Z_{(p)}$. We will denote by $\mathcal{M}_K^{ss}$ the reduced locus of $\mathcal{M}_K \otimes_{\Z_{(p)}} k$ parametrizing quadruples where the abelian variety $A$ is supersingular. This is called the supersingular locus.

Choose a geometric point $(\underline{A}, \underline{\iota}, \underline{\lambda}, \underline{\eta}^p K^p) \in \mathcal{M}^{ss}_{K}(k)$, and let $\underline{\G} = \underline{A}[p^\infty]$. This is a supersingular $p$-divisible group of height 8 and dimension 4, because $\underline{A}$ is a supersingular abelian variety of dimension 4. The action 
$\underline{\iota}$ determines an action denoted  $\iota_{\underline{\G}}$ on $\G$, and the quasi-polarization $\underline{\lambda}$ determines a principal polarization denoted $\lambda_{\underline{\G}}$ of $\G$. We can use the triple $(\underline{\G}, \iota_{\underline{\G}}, \lambda_{\underline{\G}})$ as a basepoint to define a $\GU(2,2)$ Rapoport-Zink space. This is isomorphic to the one constructed in the previous section, so will also be denoted $\mathcal{N}_{GU(2,2)}$.

Let $J$ be the subgroup of $(\End(\G) \otimes_{\Z} \Q)^\times$ of quasi-endomorphisms $g$ of $\G$ that are $\mathcal{O}_E \otimes_{\Z} \Z_p$-linear and respect the polarization: $g^*(\underline{\lambda}) = \nu(g) \underline{\lambda}$, for some $\nu(g) \in \Q_p^\times$. There is an algebraic group $I$ over $\Q$ such that $I(\Q_p) \cong J$ and $I(\Q)$ is the subgroup of $(\End(\underline{A}) \otimes_{\Z} \Q)^\times$ of quasi-endomorphisms $g$ of $\underline{A}$ that are $\mathcal{O}$ linear and respect the polarization:  $g^*(\underline{\lambda}) = \nu(g) \underline{\lambda}$ for some $\nu(g) \in \Q^\times$. The level structure $\eta^p K^p$ determines a right $K^p$-orbit of isomorphisms $I(\A_f^p) \cong G(\A_f^p)$, and so $I(\Q)$ acts on both $\mathcal{N}_{GU(2,2)}$ and $G(\A_f^p)/K^p$.  Now we may state the Rapoport-Zink uniformization theorem:

\begin{theorem}(Rapoport-Zink)
There is an isomorphism of $k$-schemes:
$$\mathcal{M}_K^{ss} \cong I(\Q) \setminus ( \mathcal{N}_{GU(2,2), red} \times G(\A_f^p) / K^p).$$

\end{theorem}

Now, as in ~\cite{GU22} and ~\cite{V}, we may use the Rapoport-Zink uniformization theorem combined with Theorem ~\ref{RZ22}   to describe the supersingular locus of the $\GU(2,2)$ Shimura variety. Note that this requires the use of Theorem 6.5 of ~\cite{V}, which continues to hold.

 \begin{theorem}
 Let $\mathcal{M}_K^{ss}$ be the supersingular locus of the $\GU(2,2)$ Shimura variety at a prime $p$ split in the relevant imaginary quadratic field, with level structure given by $K = K_pK^p$. The $k$-scheme $\mathcal{M}_K^{ss}$ has pure dimension 1. 
 
 For $K^p$ sufficiently small, all irreducible components of $\mathcal{M}_K^{ss}$ are isomorphic, over $k$, to $\mathbb{P}^1$. Any two irreducible components either intersect trivially or intersect in a single point.
  
Each irreducible component contains $p^2 + 1$ intersection points, and each intersection point is the intersection of $p^2 + 1$ irreducible components. These intersection points are precisely the superspecial points.

Further, $\mathcal{M}_K^{ss}$ has two Ekedahl-Oort strata: one consisting of the superspecial points, and the other the complement of the superspecial points.
 \end{theorem}

 \nocite{*}

	\bibliographystyle{plain}

\bibliography{testbib2}

\begin{thebibliography}{10}

\bibitem{Demazure}
Michel Demazure.
\newblock {\em Lectures on {$p$}-divisible groups}, volume 302 of {\em Lecture
  Notes in Mathematics}.
\newblock Springer-Verlag, Berlin, 1986.
\newblock Reprint of the 1972 original.

\bibitem{GU22}
Benjamin Howard and Georgios Pappas.
\newblock On the supersingular locus of the {${\rm GU}(2,2)$} {S}himura
  variety.
\newblock {\em Algebra Number Theory}, 8(7):1659--1699, 2014.

\bibitem{GSpin}
Benjamin Howard and Georgios Pappas.
\newblock Rapoport-{Z}ink spaces for spinor groups.
\newblock {\em Compos. Math.}, 153(5):1050--1118, 2017.

\bibitem{Ill}
Luc Illusie.
\newblock D\'eformations de groupes de {B}arsotti-{T}ate (d'apr\`es {A}.
  {G}rothendieck).
\newblock {\em Ast\'erisque}, (127):151--198, 1985.
\newblock Seminar on arithmetic bundles: the Mordell conjecture (Paris,
  1983/84).

\bibitem{KatO}
Toshiyuki Katsura and Frans Oort.
\newblock Supersingular abelian varieties of dimension two or three and class
  numbers.
\newblock In {\em Algebraic geometry, {S}endai, 1985}, volume~10 of {\em Adv.
  Stud. Pure Math.}, pages 253--281. North-Holland, Amsterdam, 1987.

\bibitem{Ko}
Robert~E. Kottwitz.
\newblock Points on some {S}himura varieties over finite fields.
\newblock {\em J. Amer. Math. Soc.}, 5(2):373--444, 1992.

\bibitem{Kraft}
H.~Kraft.
\newblock Kommutative algebraische p-{G}ruppen (mit {A}nwendungen auf
  p-divisible {G}ruppen und abelsche {V}ariet\"{a} ten).
\newblock {\em University of Bonn}, 1975.
\newblock manuscript.

\bibitem{KR}
Stephen~S. Kudla and Michael Rapoport.
\newblock Cycles on {S}iegel threefolds and derivatives of {E}isenstein series.
\newblock {\em Ann. Sci. \'Ecole Norm. Sup. (4)}, 33(5):695--756, 2000.

\bibitem{Lang}
Serge Lang.
\newblock {\em Algebra}, volume 211 of {\em Graduate Texts in Mathematics}.
\newblock Springer-Verlag, New York, third edition, 2002.

\bibitem{Li}
Chao Li and Yihang Zhu.
\newblock Arithmetic intersection on {GS}pin {R}apoport-{Z}ink spaces.
\newblock {\em Compos. Math.}, 154(7):1407--1440, 2018.

\bibitem{Moo}
Ben Moonen.
\newblock Group schemes with additional structures and {W}eyl group cosets.
\newblock In {\em Moduli of abelian varieties ({T}exel {I}sland, 1999)}, volume
  195 of {\em Progr. Math.}, pages 255--298. Birkh\"auser, Basel, 2001.

\bibitem{EO}
Frans Oort.
\newblock A stratification of a moduli space of abelian varieties.
\newblock In {\em Moduli of abelian varieties ({T}exel {I}sland, 1999)}, volume
  195 of {\em Progr. Math.}, pages 345--416. Birkh\"auser, Basel, 2001.

\bibitem{Min}
Frans Oort.
\newblock Minimal {$p$}-divisible groups.
\newblock {\em Ann. of Math. (2)}, 161(2):1021--1036, 2005.

\bibitem{Pries}
Rachel Pries.
\newblock A short guide to {$p$}-torsion of abelian varieties in characteristic
  {$p$}.
\newblock In {\em Computational arithmetic geometry}, volume 463 of {\em
  Contemp. Math.}, pages 121--129. Amer. Math. Soc., Providence, RI, 2008.

\bibitem{RZ}
M.~Rapoport and Th. Zink.
\newblock {\em Period spaces for {$p$}-divisible groups}, volume 141 of {\em
  Annals of Mathematics Studies}.
\newblock Princeton University Press, Princeton, NJ, 1996.

\bibitem{RTW}
Michael Rapoport, Ulrich Terstiege, and Sean Wilson.
\newblock The supersingular locus of the {S}himura variety for {${\rm
  GU}(1,n-1)$} over a ramified prime.
\newblock {\em Math. Z.}, 276(3-4):1165--1188, 2014.

\bibitem{AoQF}
Goro Shimura.
\newblock {\em Arithmetic of quadratic forms}.
\newblock Springer Monographs in Mathematics. Springer, New York, 2010.

\bibitem{Viehmann}
Eva Viehmann.
\newblock Moduli spaces of {$p$}-divisible groups.
\newblock {\em J. Algebraic Geom.}, 17(2):341--374, 2008.

\bibitem{V}
Inken Vollaard.
\newblock The supersingular locus of the {S}himura variety for {${\rm
  GU}(1,s)$}.
\newblock {\em Canad. J. Math.}, 62(3):668--720, 2010.

\bibitem{VW}
Inken Vollaard and Torsten Wedhorn.
\newblock The supersingular locus of the {S}himura variety of {${\rm
  GU}(1,n-1)$} {II}.
\newblock {\em Invent. Math.}, 184(3):591--627, 2011.

\bibitem{Windows}
Thomas Zink.
\newblock Windows for displays of {$p$}-divisible groups.
\newblock In {\em Moduli of abelian varieties ({T}exel {I}sland, 1999)}, volume
  195 of {\em Progr. Math.}, pages 491--518. Birkh\"auser, Basel, 2001.

\end{thebibliography}

\end{document}